\documentclass{article}
\usepackage[margin=1in]{geometry}
\usepackage{amsthm}
\newtheorem{theorem}{Theorem}
\newtheorem{definition}{Definition}
\newtheorem{lemma}{Lemma}

\usepackage{graphicx}
\usepackage{lipsum}
\usepackage{amsfonts}
\usepackage{epstopdf}
\usepackage{algorithmic}
\ifpdf
  \DeclareGraphicsExtensions{.eps,.pdf,.png,.jpg}
\else
  \DeclareGraphicsExtensions{.eps}
\fi

\usepackage{amsopn}

\usepackage[american]{babel}
\usepackage{tikz}
\usepackage{tikzscale}
\usepackage{subfig,tabu}
\usetikzlibrary{external}
\usepackage{pgfplots}

\usepackage{amssymb, palatino, geometry}
\usepackage{amsfonts, amsmath, latexsym}
\usepackage{epsfig,graphicx,xcolor}
\usepackage[algo2e,linesnumbered,vlined,lined,boxed]{algorithm2e} 
\usepackage{url}                
\usepackage{version}            
\usepackage[titletoc,title]{appendix}

\usepackage[normalem]{ulem} 

\graphicspath{{../../figures/}{../images/}}        

\newtheorem{assumption}{Assumption}

\newcommand{\Z}{\mathbb Z}           
\newcommand{\R}{\mathbb R}           

\newcommand{\mini}{\mathop{\mbox{minimize}}}

\newcommand{\st}{\mbox{subject to }}
\newcommand{\bmath}[1]{\mbox{\boldmath $ #1 $}}
\newcommand{\be}{\begin{equation}}
\newcommand{\ee}{\end{equation}}
\newcommand{\bea}{\begin{eqnarray}}
\newcommand{\eea}{\end{eqnarray}}
\newcommand{\dps}{\displaystyle}

\newcommand{\bvec}{\left(\begin{array}{c}}
\newcommand{\evec}{\end{array}\right)}
\newcommand{\bsub}{\begin{subequations}}
\newcommand{\esub}{\end{subequations}}

\newcommand{\defined}{=}
\newcommand{\T}{{\operatorname{T}}}
\newcommand{\bi}{\bmath{i}}
\newcommand{\cU}{\mathcal{U}}

\definecolor{darkgreen}{rgb}{0,0.5,0}

\newcommand{\conep}[1]{\mathrm{cone}\big(#1\big)}

\newcommand{\convp}[1]{\mathrm{conv}\left(#1\right)} 
\newcommand{\intp}[1]{\mathrm{int}\left( #1 \right)}

\DeclareMathOperator*{\argmin}{arg\,min}
\newcommand{\thmref}[1]{Theorem~\ref{thm:#1}}
\newcommand{\figref}[1]{Figure~\ref{fig:#1}}

\newcommand{\lemref}[1]{{Lemma~\ref{lem:#1}}}
\newcommand{\defref}[1]{{\rm Definition~\ref{def:#1}}}
\newcommand{\asref}[1]{{\rm Assumption~\ref{ass:#1}}}

\newcommand{\algref}[1]{{\rm Algorithm~\ref{alg:#1}}}
\newcommand{\secref}[1]{Section~\ref{sec:#1}}
\newcommand{\appref}[1]{Appendix~\ref{app:#1}}
\newcommand{\propref}[1]{Proposition~\ref{prop:#1}}

\newcommand{\linerefa}[1]{Line~\ref{line:#1}} 

\newcommand{\tabref}[1]{{\rm Table~\ref{table:#1}}}
\newcommand{\tabrefs}[2]{Tables~\ref{table:#1}--\ref{table:#2}}
\newcommand{\biminusj}{\bmath{i}_{i_j-}}

\newcommand{\ds}{\displaystyle}
\newcommand{\floor}[1] {\lfloor #1 \rfloor}

\newcommand{\code}[1]{{\rm \textsf{#1}}}
\newcommand{\prob}[1]{{\rm \texttt{#1}}}
\allowdisplaybreaks

\usepackage{placeins}

\newtheorem{problem}{Problem}
\newtheorem{proposition}{Proposition}[section]

\begin{document}

\title{A Method for Convex Black-Box Integer Global Optimization}


\author{Jeffrey Larson \and Sven Leyffer \and Prashant Palkar \and \mbox{Stefan M.~Wild}}



\maketitle

\begin{abstract}
  We study the problem of minimizing a convex function on a nonempty, finite subset of the integer lattice
  when the function cannot be evaluated at noninteger points. 
  We propose a new underestimator that does not require access to
  (sub)gradients of the objective but, rather, uses secant linear functions that interpolate the
  objective function at previously evaluated points. These linear mappings are shown to
  underestimate the objective in disconnected portions of the domain. Therefore, the union of these
  conditional cuts provides a nonconvex underestimator of the objective. 
  We propose an algorithm that alternates between updating the underestimator
  and evaluating the objective function.
  We prove that the algorithm converges to a global minimum of the objective 
  function on the
  feasible set. We present two
  approaches for representing the underestimator and compare their
  computational effectiveness. We also compare implementations of our algorithm
  with existing methods for minimizing functions on a subset of the
  integer lattice. We discuss the 
  difficulty of this problem class
  and provide insights into why a computational proof of optimality is
  challenging even for moderate problem sizes.
\end{abstract}

\section{Introduction}\label{S:intro}
We study the problem of minimizing a convex function on a finite subset of the integer lattice.
In particular, we consider problems of the form
\begin{equation}\label{eq:MIP-DFO} 
  \mini_x \; f(x) \quad \st \; x \in \Omega \subset \Z^{n}.
\end{equation}
We first define what it means for $f$ to be convex on a set $\Omega \subset \Z^n$.
\begin{definition}[Convexity on Integer Subsets]\label{def:convexity}
  A function $f$ is \emph{convex on $\Omega$} if for $x \in \Omega$ and
  any $p$ points $y^i \in \Omega$ satisfying
  $x = \sum_{i=1}^{p} \lambda_i y^i$
  with $\lambda_i \in [0,1]$ and $\sum_{i=1}^{p}\lambda_i = 1$, then
  $f(x) \le \sum_{i=1}^{p} \lambda_i f(y^i)$.
\end{definition}
We make the following assumption about problem \eqref{eq:MIP-DFO}.
\begin{assumption}\label{ass:1}
$f$ is convex on $\Omega$,
  $\Omega \subset \Z^n$ is nonempty and bounded,
  and $f$ cannot be evaluated at $x \notin \Omega$.
\end{assumption}

Because we assume that $f$ cannot be evaluated at noninteger points, 
problem \eqref{eq:MIP-DFO} can be referred to as a 
convex 
optimization problem with {\em unrelaxable integer
constraints} \cite{taxonomy15}. 
We note that 
$\Omega$ need not contain all integer points in its 
convex hull (i.e., our approach allows for situations where $\convp{\Omega} 
\cap \Z^n \not = \Omega$).
Problem 
\eqref{eq:MIP-DFO} can also be viewed as minimizing an
\emph{integer-convex}\footnote{Although \cite{Hemmecke10} considers
integer convexity only for polynomials, this definition can be 
applied to more general classes of functions over the sets considered here.}
function \cite[Definition 15.2]{Hemmecke10} over a nonempty finite subset of
$\Z^n$.

Admittedly, it is rare to know that $f$ is convex when $f$
is not given in closed form (although one may be able to detect convexity
\cite{NanjingJian2014} or estimate a probability that $f$ is convex on
the finite domain $\Omega$~\cite{Jian2019}). Nevertheless, 
studying the convex
case is
important because we are unaware of any method (besides complete enumeration) for obtaining exact solutions
to \eqref{eq:MIP-DFO} when $f$ (convex or otherwise) cannot be evaluated at
noninteger points. 

One example where an
objective is not given in closed form but is known to be convex arises in the combinatorial optimal control of
partial differential equations (PDEs).
For example, Buchheim et al.~\cite[Lemma 2]{Buchheim2018} show
that the solution operator of certain semilinear elliptic PDEs is a convex function of the controls
provided that the nonlinearities in the PDE and boundary conditions are concave and
nondecreasing. Thus, any linear function of the states
of the PDE (e.g., the $\max$-function) is a convex function
of the controls when the states are eliminated. The authors of \cite{Buchheim2018}
propose using adjoint information to compute subgradients of the 
(continuous relaxation of the) objective, but an alternative
would be to consider a derivative-free approach.

We consider only pure-integer 
problems of the form \eqref{eq:MIP-DFO}; however, our developments are equally applicable to
the mixed-integer case
\begin{equation}\label{eq:MixedIP-DFO} 
  \mini_{x,y} \; F(x,y) \quad \st \; (x,y) \in \Omega \times \Psi \subset \Z^{n} \times \R^m
\end{equation}
provided $F$ is convex on  $\Omega \times \Psi$. 
If we define the function 
\[ f(x) \defined \min_{y \in \Psi} F(x,y) \]
and if $f$ is well defined over $\Omega$,
then \eqref{eq:MixedIP-DFO} can be solved by minimizing $f$ on $\Omega \subset \Z^n$,
where each evaluation $f(x)$ requires an optimization of the continuous
variables $y$ for a fixed $x$. (When there is no additional
information on $F$ and $\Psi$, each continuous optimization problem 
may be difficult to solve.) Because many of the results below rely only on the convexity of
$f$ and not the discrete nature of $\Omega$, much of the analysis below readily
applies to the mixed-integer case. 

We are especially interested in problems where the cost of evaluating $f$ is 
large. 
Problems of the form \eqref{eq:MIP-DFO} or \eqref{eq:MixedIP-DFO} where the
objective is expensive to evaluate and some integer constraints are unrelaxable arise in a range of
simulation-based optimization problems. For example, the optimal design of
concentrating solar power plants gives rise to computationally expensive simulations for 
each set of design parameters \cite{cspopt11}. Furthermore, 
some of the design parameters (e.g., the number of panels on the power plant
receiver) cannot be relaxed to noninteger values. 
Similar problems arise when tuning codes to run on 
high-performance computers~\cite{Multiobj13}. In this case, $f(x)$ may be the
memory footprint of a code that is compiled with settings $x$, which 
can correspond to decisions such as loop unrolling or tiling factors 
that do not have meaningful noninteger values. Optimal material design problems
may also constrain the choice of atoms to a finite set, resulting in unrelaxable integer constraints;
see~\cite{Graf2017} for a derivative-free optimization 
algorithm designed explicitly for such a
problem. 

Motivated by such applications, we develop a method that will
certifiably converge to the solution of \eqref{eq:MIP-DFO} under \asref{1}
without access to $\partial f$. Using only evaluations of $f$, we construct
\emph{secants}, which are linear functions that interpolate $f$ at a set of
$n+1$ points. These secant functions underestimate $f$ in certain parts of
$\Omega$. We use these
secants to define \emph{conditional cuts} that are
valid in disconnected portions of the domain. The complete set of secants 
and the conditions that describe when they are valid are used to
construct an underestimator of a convex $f$.
While access to $\partial f$ (i.e., a (sub)gradient of a continuous relaxation of $f$) 
could strengthen such an underestimator, we do not
address such considerations in this paper. 

Solving \eqref{eq:MIP-DFO} under \asref{1} without access to
$\partial f$ poses a number of theoretical
and computational challenges. Because the integer constraints are
unrelaxable, one cannot apply traditional branch-and-bound approaches. In
particular, model-based continuous derivative-free methods would require evaluating
the objective at noninteger points to ensure convergence for the 
continuous relaxation of \eqref{eq:MIP-DFO}. 
In addition, other
traditional techniques for mixed-integer optimization---such as Benders decomposition
\cite{Geoffrion1972} or outer approximation
\cite{Duran1986,Fletcher1994}---cannot be used to solve
\eqref{eq:MIP-DFO} when $\partial f$ is unavailable.
Since we know of no method (other than complete enumeration) for obtaining
global minimizers of \eqref{eq:MIP-DFO} under \asref{1}, we know of no
potential algorithm to address this problem when a (sub)gradient is
unavailable. 

We make three contributions in this paper: (1) we develop a new
underestimator for convex functions on subsets of the integer lattice that is based
solely on function evaluations; (2)
we present an algorithm that alternates between updating 
this underestimator and evaluating the objective in order to identify a global
solution of problem \eqref{eq:MIP-DFO} under \asref{1}; and
(3) we show empirically that
certifying global optimality 
in such cases 
is a challenging
problem. In our experiments, we are unable to prove optimality for many problems 
when $n\geq 5$, and we provide insights into why a proof of
optimality remains computationally challenging.

\paragraph{Outline.}
\secref{background} surveys recent methods for addressing \eqref{eq:MIP-DFO}. 
\secref{cuts} introduces
valid conditional cuts using only the function values of a convex objective
and discusses the theoretical properties of these cuts.
\secref{algm} presents an algorithm for solving \eqref{eq:MIP-DFO}
and shows that this algorithm identifies a global minimizer of \eqref{eq:MIP-DFO} under \asref{1}.
\secref{implementing} considers two approaches for formulating the underestimator
and presents the method \code{SUCIL}---secant 
underestimator of convex functions on the integer lattice. 
\secref{numerics} provides detailed numerical studies for implementations of \code{SUCIL} on a set of convex problems.
\secref{chall} discusses many of the challenges in obtaining global solutions to \eqref{eq:MIP-DFO}.

\section{Background}\label{sec:background}

Developing methods to solve \eqref{eq:MIP-DFO} without access to derivatives of
$f$ is an active area of research.
Most methods address general (i.e., nonconvex) functions $f$, and
heuristic approaches are commonly adopted to handle integer decision
variables for such derivative-free optimization problems. 
For example, the method in \cite{Mueller2013a} 
rounds noninteger components of candidate points to the nearest
feasible integer values. The method's asymptotic convergence results are based on the
inclusion of points drawn uniformly from the finite domain (and rounding noninteger values
as necessary). 

Integer-constrained pattern-search
methods~\cite{Abramson2009,Audet2001} generalize their
continuous counterparts.
These modified pattern-search methods can be shown to converge to 
{\em mesh-isolated minimizers}: points with function values that are better than all 
neighboring points on the integer lattice. Unfortunately, such mesh-isolated
minimizers can be arbitrarily far from a global minimizer, even
when $f$ is convex; see \cite[Fig.~2]{Abhishek2010} for an example of such a function $f$. 
Other methods that converge to mesh-isolated minimizers include
direct-search methods that update the integer variables via a 
local search \cite{Liuzzi2011,Liuzzi2015} 
and mesh adaptive direct-search methods adapted
to address discrete and granular variables (i.e., those that have a controlled number
of decimals)~\cite{Abramson2014,AlDT2018}.
The direct-search method in \cite{Garcia-Palomares2019} accounts for integer 
constraints
by constructing a set of directions that have a nonnegative span
of $\R^n$ and that ensure that all evaluated points will be integer valued. This
method is shown to converge to a type of stationary point
that, even in the convex case, may not be a global minimizer.
See \cite{Newby2015} for various definitions of local minimizers of
\eqref{eq:MIP-DFO} and a discussion of associated properties.
The 
BFO method in \cite{Porcelli2017} has a recursive
step that explores points near the current iterate by fixing each of the
discrete variables to its value plus or minus a step-size parameter.

\begin{table}[ht]
  \begin{minipage}[b]{0.34\linewidth}
    \centering
    \includegraphics[width=\linewidth]{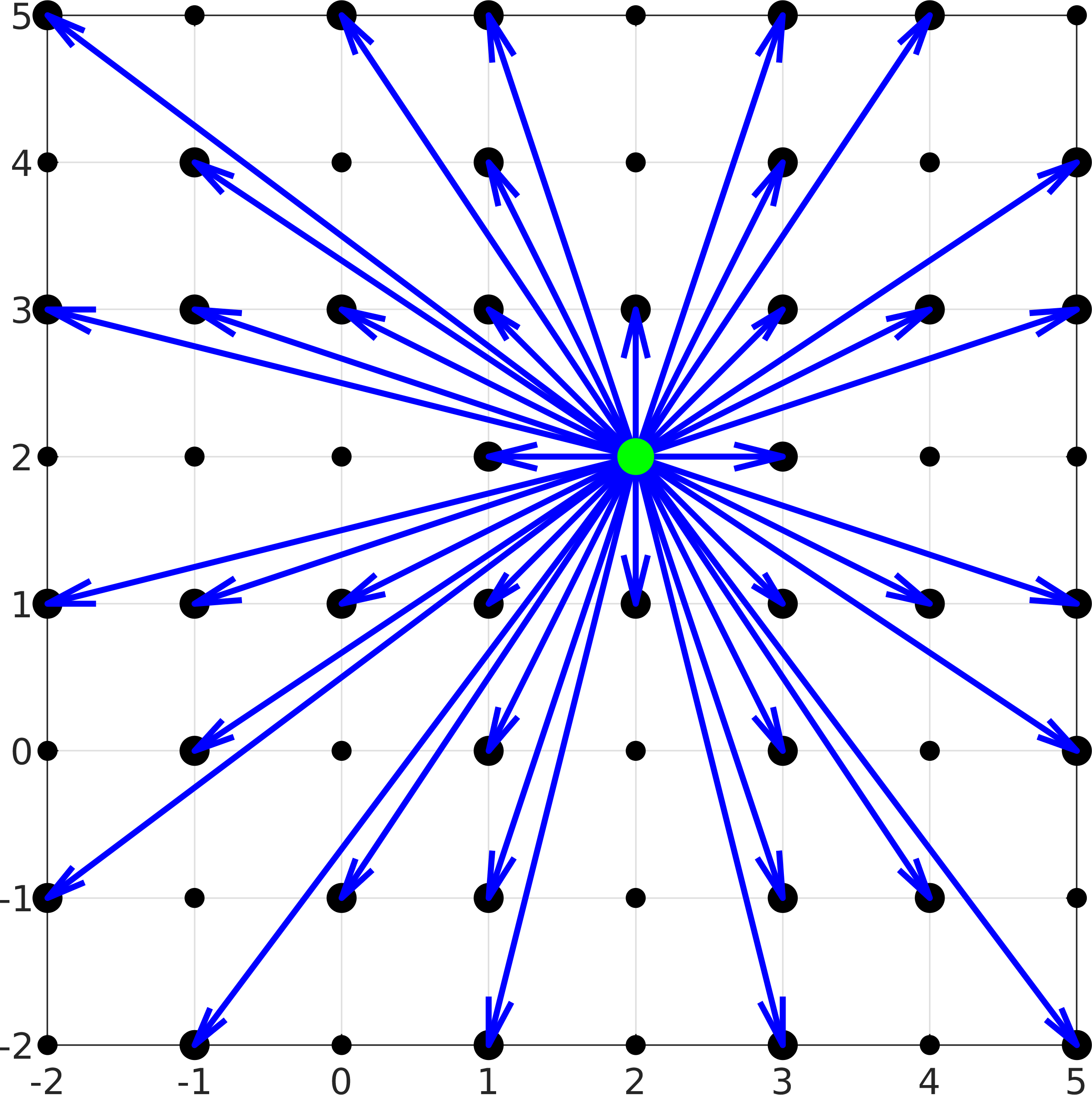}
    \captionof{figure}{Primitive directions emanating from $(2,2)$ in the domain $\Omega = [-2,5]^2 \cap 
    \Z^2$.  \label{fig:PrimDirs}}
  \end{minipage}
  \hfill
  \begin{minipage}[b]{0.61\linewidth}%
    \small
    \centering
  \begin{tabular}{|c||r|r|r|r|r|r|r|r|r|r|r|c|}
  \hline
  & \multicolumn{2}{c|}{$n=2$} &
  \multicolumn{ 2}{c|}{$n=3$} & \multicolumn{ 2}{c|}{$n=4$} & \multicolumn{ 2}{c|}{$n=5$}\\ \hline
  $k$ & $\lvert \Omega \rvert$ &
  $\#$ & $\lvert \Omega \rvert$ &
  $\#$ & $\lvert \Omega \rvert$ &
  $\#$ & $\lvert \Omega \rvert$ &
  $\#$\\ \hline \hline
  1  & 9   & 8  & 27   & 26   & 81    & 80     & 243    & 242 \\ \hline
  2  & 25  & 16 & 125  & 98   & 625   & 544    & 3,125   & 2,882 \\ \hline
  3  & 49  & 32 & 343  & 290  & 2,403  & 2,240   & 16,807  & 16,322 \\ \hline
  4  & 81  & 48 & 729  & 578  & 6,561  & 5,856   & 59,049  & 55,682 \\ \hline
  \end{tabular}
    \caption{Number of primitive directions, $\# = \lvert \mathcal{N}(x_c,1)
    \rvert$, that emanate from the origin $x_c$ of the domain ${\Omega=\left[
    -k,k \right]^n \cap \Z^n}$ and that correspond to points in
    $\Omega$.\label{table:directions}}
  \end{minipage}
\end{table}

The method 
in \cite{Liuzzi2018} 
uses line searches over a set of {\em primitive directions}, that is,  
a set of scaled
directions $D$ where no vector $d_i \in D$ is a positive 
multiple of a different $d_j \in D$. 
This method explores a discrete set of
directions around the current iterate until finding a
local minimum $x_c$ in a $\beta$-neighborhood, defined as
$\mathcal{N}(x_c,\beta) \defined \{ x_c + \alpha d \in \Omega: d\in D, 
\alpha \in \mathbb{N}, \alpha \le \beta\}$
for $\beta \in \mathbb{N}$.
Although the authors of \cite{Liuzzi2018} target nonconvex objectives, their approach
will converge to a global minimum $x_c$ of a convex objective $f$ 
if all points in $\mathcal{N}(x_c,1)$ are evaluated. 
\figref{PrimDirs} illustrates such a discrete 1-neighborhood.
Unfortunately, $|\mathcal{N}(x_c,1)|$
can be large; see \tabref{directions}.

Model-based methods approximate objective functions on the integer lattice
by using surrogate models;
see,
for example, \cite{HQE08}, \cite{Rashid12}, and \cite{costa2014rbfopt}. The 
surrogate model
is used to determine points where the objective should be evaluated; the
model is typically refined after each objective evaluation. The methodology in 
\cite{costa2014rbfopt}
specifically uses radial basis function surrogate models and does automatic
model selection at each iteration. 
Mixed-integer nonlinear optimization solvers are used to minimize the surrogate 
to obtain the next integer point for evaluation.
The model-based methods of
\cite{Mueller2014a,muller2016miso,Mueller2013}
modify the sampling strategies and local searches typically used to solve
continuous objective versions. The approaches in
\cite{Mueller2014a,muller2016miso} restart when a suitably 
defined local
minimizer is encountered, continuing to evaluate the objective until the
available budget of function evaluations is exhausted. These model-based
methods differ in the initial sampling method, the type of surrogate
model, and the sampling strategy used to select the next points to be evaluated. 
See \cite{BB2017}, for a survey and taxonomy of continuous and integer
model-based optimization approaches. 

In a different line of research, Davis and Ierapetritou~\cite{Davis:2009}
propose a branch-and-bound framework to address binary variables; a
solution to the relaxed nonlinear subproblems is obtained via a combination of 
global kriging models and local surrogate models. 
Similarly, Hemker et al.~\cite{Hemker08} replace the black-box portions of the 
objective function
(and constraints) by a stochastic surrogate; the resulting
mixed-integer nonlinear programs are solved by branch and bound. Both 
approaches assume that the integer constraints are relaxable.

\section{Underestimator of Convex Functions on the Integer Lattice}\label{sec:cuts}

To construct an underestimator of a convex objective function $f$, 
we now discuss secant functions, which are linear mappings that 
interpolate $f$ at $n+1$ points. We provide conditions for where 
these cuts will underestimate $f$. We then discuss necessary conditions on the set
of evaluated points so that if all possible secants are constructed, these
conditional cuts underestimate $f$ on the domain $\Omega$. 
As we will see in \secref{algm},
this underestimator is essential for obtaining a global minimizer of
\eqref{eq:MIP-DFO} under \asref{1}.
Throughout this section, $X \subseteq \Omega$ denotes a set of at least
$n+1$ points in $\Omega$ at which the objective function $f$ has been evaluated.

\subsection{Secant Functions and Conditional Cuts}

Constructing a secant function requires a set of $n+1$ interpolation
points where $f$ has been evaluated.
To define a secant function for $f$, we introduce a multi-index $\bi$ of
$n+1$ distinct indices, $1 \leq i_1 < \ldots < i_{n+1} \leq \left| X \right|$, as
$\bi \defined (i_1,\ldots,i_{n+1})$.
With a slight abuse of notation, we will refer to elements $i_j \in \bi$.

Given the set of points $X^{\bi} \defined \left\{ x^{i_j} \in X \subseteq \Omega : i_j \in \bi \right\}$, 
we construct the secant function 
\begin{equation*}
  m^{\bi}(x) \defined (c^{\bi})^\T x + b^{\bi},
\end{equation*}
where the coefficients $c^{\bi}\in \R^n$ and $b^{\bi}\in \R$ are the solution to the linear system
\begin{equation}\label{eq:interp}
    \left[ \bar{X}^{\bi} \; e \right]
   \left[ \begin{array}{l}
     c^{\bi} \\
     b^{\bi} 
   \end{array} \right]
   =
     f^{\bi},
\; \text{where} \;
\bar{X}^{\bi} \defined 
  \left[ 
    \begin{array}{c}
      (x^{i_1})^\T\\
      \vdots    \\
      (x^{i_{n+1}})^\T 
    \end{array} 
  \right], 
\;
e \defined 
  \left[
    \begin{array}{c}
      1\\
      \vdots \\
      1
    \end{array} 
  \right], \; \mbox{and} \;
f^{\bi} \defined
  \left[
    \begin{array}{c}
      f(x^{i_1}) \\
      \vdots \\
      f(x^{i_{n+1}})
    \end{array} 
  \right].
\end{equation}
The secant function $m^{\bi}$ is unique provided that the set
$X^{\bi}$ is \emph{poised}, which we now define.
\begin{definition}\label{def:poised}
  The set of points $X^{\bi}$ is \emph{poised} if the matrix $\left[ \bar{X}^{\bi} \; e  \right]$ is nonsingular.
\end{definition}
Note that \defref{poised} is equivalent to $X^{\bi}$ being affinely independent. 
We now show that the secant function $m^{\bi}$ underestimates $f$
in certain polyhedral cones, namely, the cones 
\begin{equation}\label{eq:U_def}
  \cU^{\bi} \defined \bigcup\limits_{i_j \in \bi} \conep{x^{i_j} - X^{\bi}}, 
\end{equation}
where
\begin{equation}\label{eq:cone_def}
  \conep{x^{i_j} - X^{\bi}} \defined 
  \{x^{i_j}+\sum_{l=1,l \neq j}^{n+1} \lambda_l (x^{i_j} - x^{i_l}) :\, 
  i_j \in \bi, i_l \in \bi, \lambda_l \geq 0 \}.
\end{equation}

\begin{lemma}[Conditional Cuts]\label{lem:cone}
  If $f$ is convex on $\Omega$ 
  and $X^{\bi}$ $\subseteq \Omega$ is poised, then the unique linear mapping
  $m^{\bi}$ satisfying $m^{\bi}(x^{i_j}) = f(x^{i_j})$ for each $i_j \in \bi$ 
  satisfies $m^{\bi}(x) \le f(x)$ for all $x \in \cU^{\bi} \bigcap \Omega$.
\end{lemma}

\begin{proof}
The uniqueness of the linear mapping follows directly from the affine 
independence guaranteed by \defref{poised} for poised $X^{\bi}$.

Let $x$ be a point in $\conep{x^{i_j}- X^{\bi}} \bigcap \Omega$ for arbitrary
  $x^{i_j} \in X^{\bi}$. By \eqref{eq:cone_def}, 
  \begin{equation}\label{eq:x_def}
    x = x^{i_j} + \sum_{l=1,l \neq j}^{n+1} \lambda_l \left(x^{i_j} - x^{i_l}\right),
  \end{equation}
  with $\lambda_l \ge 0$ (for $l=1,\ldots, n+1; l\neq j$). 
  Rearranging 
\eqref{eq:x_def}
  yields
  \begin{equation*}
    x^{i_j} = \frac{1}{1 + \sum_{k=1,k \neq j}^{n+1} \lambda_k} x + 
\frac{1}{1 + \sum_{k=1,k \neq j}^{n+1} 
\lambda_k} \sum_{l=1,l \neq j}^{n+1} \lambda_l x^{i_l},
  \end{equation*}
  showing that $x^{i_j}$
can be expressed as a convex combination of $\{x\} \cup \{ x^{i_l}: l = 1,\ldots,n+1;l\neq j\}$, 
all of which are points in $\Omega$. 
  Therefore, by convexity of $f$ on $\Omega$ (see \defref{convexity}),
  \begin{equation*}
    f(x^{i_j}) \le \frac{1}{1 + \sum_{k=1,k \neq j}^{n+1} \lambda_k} f(x) + 
\frac{1}{1 + \sum_{k=1,k \neq j}^{n+1} 
\lambda_k} \sum_{l=1,l \neq j}^{n+1} \lambda_l f\left( x^{i_l} \right).
  \end{equation*}
  Solving for $f(x)$ and using the fact that $m^{\bi}$ interpolates $f$ at 
  points in $X^{\bi}$, we obtain
  \begin{align*}
    f(x) &\ge \left(1 + \sum_{k=1,k \neq j}^{n+1}  \lambda_k\right) f(x^{i_j}) 
- \sum_{l=1,l \neq j}^{n+1}  \lambda_l f(x^{i_l})\\
    &= \left(1 + \sum_{k=1,k \neq j}^{n+1} \lambda_k\right) m^{\bi}(x^{i_j}) - 
\sum_{l=1,l \neq j}^{n+1} \lambda_l m^{\bi}(x^{i_l})\\
    &= m^{\bi}(x^{i_j}) + \sum_{l=1,l \neq j}^{n+1} \lambda_l 
\left(m^{\bi}(x^{i_j}) -  m^{\bi}(x^{i_l})\right)\\
    &= m^{\bi}(x),
  \end{align*}
  where the last equality holds by \eqref{eq:x_def} and the linearity of $m^{\bi}$. 
  Because $x$ is an arbitrary point in $\conep{x^{i_j} - X^{\bi}}$ for arbitrary $x^{i_j}$, the result is shown.
\end{proof}


We now prove that the 
cones in $\cU^{\bi}$ do not intersect 
when $X^{\bi}$ is poised. We note that the following result holds for any
poised set $X^{\bi} \subset \R^n$.
\begin{lemma}[A Point Is In At Most One Cone] \label{lem:only_one_cone}
  If $X^{\bi}$ is poised, no point $x \in \R^n$ satisfies $x \in
  \conep{x^{i_j} - X^{\bi}}$ and $x \in \conep{x^{i_k} - X^{\bi}}$ for 
$x^{i_j},x^{i_k} \in X^{\bi}$ and $x^{i_j}
  \neq x^{i_k}$.
\end{lemma}
\begin{proof}
  Let $x^{i_1}$ and $x^{i_2}$ be different, but otherwise arbitrary, points in $X^{\bi}$. 
  To arrive at a contradiction, suppose that there 
exists $x \in \conep{x^{i_1}-X^{\bi}} \cap \conep{x^{i_2}-X^{\bi}}$. That is,
    $x = x^{i_1} + \sum_{l=2}^{n+1} \lambda_l \left(x^{i_1} - x^{i_l} \right)$ 
and 
    $x = x^{i_2} + \sum_{l=1,l \neq 2}^{n+1} \sigma_l \left(x^{i_2} - 
x^{i_l}\right)$
  for $\lambda_l \ge 0$ ($l \in \left\{2,\ldots,n+1 \right\}$)
  and $\sigma_l \ge 0$ ($l \in \left\{1,3,\ldots,n+1 \right\}$).  Subtracting 
these two expressions yields
  \begin{align}\label{eq:sum}
    0&= x^{i_1} - x^{i_2} + \sum_{l=2}^{n+1} \lambda_l (x^{i_1} - x^{i_l}) - 
\sum_{l=1,l \neq 2}^{n+1} \sigma_l (x^{i_2} - x^{i_l})\nonumber\\
    &= x^{i_1} - x^{i_2} + \sum_{l=2}^{n+1} \lambda_l x^{i_1} - 
\sum_{l=2}^{n+1} \lambda_l x^{i_2} + \sum_{l=2}^{n+1} \lambda_l x^{i_2} - 
\sum_{l=2}^{n+1} \lambda_l x^{i_l} - \sum_{l=1,l \neq 2}^{n+1} \sigma_l (x^{i_2} 
- x^{i_l})\nonumber\\
    &= \left(1 + \sum_{l=2}^{n+1} \lambda_l\right)\left(x^{i_1} - 
x^{i_2}\right) - \sum_{l=2}^{n+1} \lambda_l \left(x^{i_l} - x^{i_2}\right) + 
\sum_{l=1,l \neq 2}^{n+1} \sigma_l \left(x^{i_l} - x^{i_2}\right)\nonumber\\
    &= \left(1 + \sigma_1 + \sum_{l=2}^{n+1} 
\lambda_l\right)\left(x^{i_1} - x^{i_2}\right) + \sum_{l=3}^{n+1} \left( 
\sigma_l- \lambda_l \right) \left(x^{i_l} - x^{i_2}\right).
  \end{align}
  Since $X^{\bi}$ is a poised set, \defref{poised} ensures that the vectors 
$\left\{ x^{i_l}-x^{i_2} : \, i_l \in \bi,
  i_l \neq i_2 \right\}$ are linearly independent. Hence the dependence 
relation in \eqref{eq:sum} can be satisfied only if the coefficient on
  $(x^{i_1}-x^{i_2})$ vanishes. That is,
  \[
 1 + \sigma_1 + \sum_{l=2}^{n+1}\lambda_l  = 0,
  \]
  which contradicts $\lambda_l \ge 0$ (for $l=2,\ldots, n+1$) and $\sigma_1 \ge 
0$.
  Since $x^{i_1}$ and $x^{i_2}$ were arbitrary points in $X^{\bi}$, the result 
is shown. 
\end{proof}

For each poised set $X^{\bi}$, \lemref{cone} ensures that the secant function
$m^{\bi}$ 
underestimates $f$ 
on $\Omega \cap \cU^{\bi}$. 
We can therefore underestimate $f$ via a model that consists of the
pointwise maximum of the underestimators for which the point is in
$\cU^{\bi}$ for some poised set $X^{\bi}$ of
previously evaluated points. Minimizing this nonconvex model on the
integer $\Omega$ can then provide 
a lower bound for the global minimum of $f$ on $\Omega$.

\begin{figure}
  \begin{center}
    \includegraphics[width=1.8in]{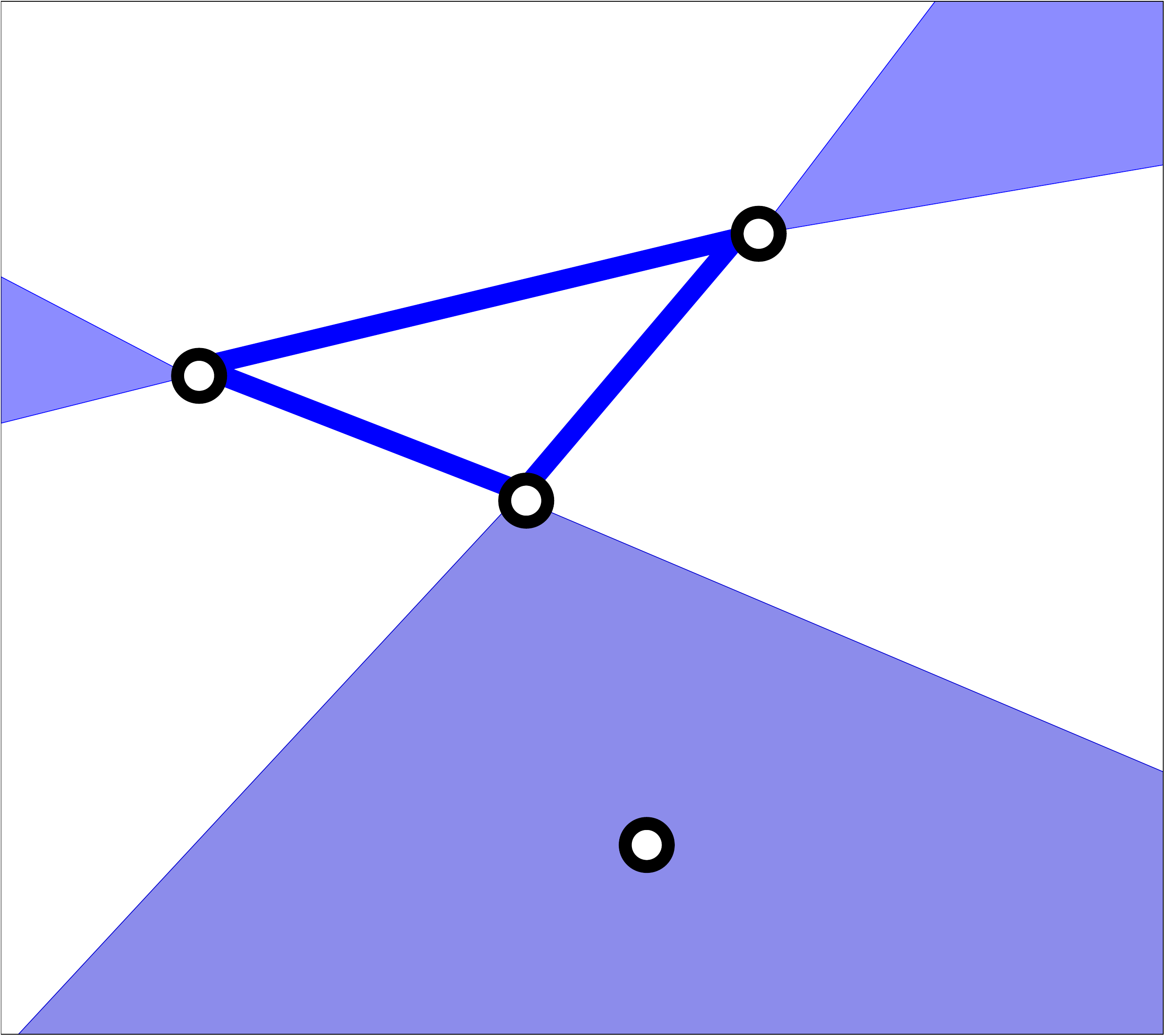}
    \hspace{20pt}
    \includegraphics[width=1.8in]{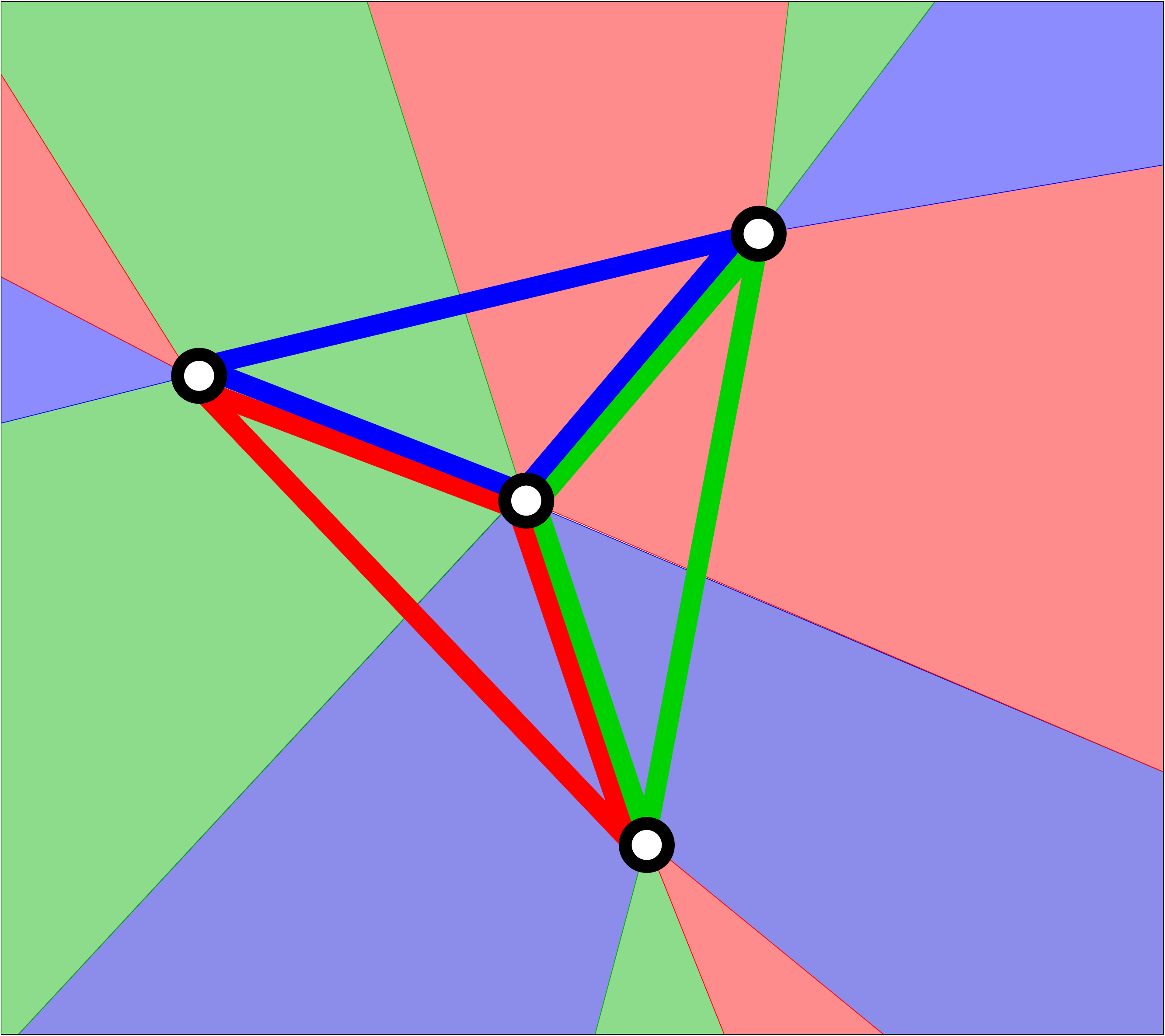}
    \caption{Illustration of areas in $\R^2$ where conditional cuts are valid.
    Left shows the regions of the domain where the secant through
    three points (the vertices of the blue triangle) will underestimate $f$.
    Right shows that one point in the interior of $n+1$ points is sufficient to
    underestimate $f$. The conditional cuts correspond to the $n+1$ points in the
    triangle of the same color.\label{fig:Cuts2D}}
  \end{center}
\end{figure}

\figref{Cuts2D} shows a two-dimensional example of points that produce such secant functions
and the regions in which they will underestimate any convex function $f$.
For the $n+2$ points (circles), we consider three poised sets indicated by
triangles linking $n+1$ points. The left image shows a poised set
(blue line triangle), and the three cones (shaded blue area) in which the secant
through these points is a valid underestimator. The right image shows
that conditional cuts using $n+2$ points can cover all of
$\R^n$. 

\subsection{Lower Bound for $f$}
We now describe an optimization problem whose solution provides a lower bound
for $f$ on $\Omega$. Let $W(X)$ denote the
set of all multi-indices corresponding to poised subsets of
$X$:
\begin{equation}\label{eq:WX_def}
  W(X) \defined \left\{ \bi: X^{\bi} \subseteq X, X^{\bi} \mbox{ is poised} \right\}.
\end{equation}
If $f$ has been evaluated at every point in $X$, we can construct a secant function 
$(c^{\bi})^\T x + b^{\bi}$ 
interpolating $f$ on $X^{\bi}$ for every
multi-index $\bi \in W(X)$.
We then collect all such conditional cuts in the piecewise linear program
\begin{equation} \label{model:PLP} \tag{PLP}
  \begin{aligned}
    \dps \mini_{x,\eta} \; & \eta \\
    \st  \; & \dps \eta \geq (c^{\bi})^\T x + b^{\bi}, \text{ for all } \bi \in W(X) 
\text{ with } x \in \cU^{\bi}\\
    \; & \dps x \in \Omega,
  \end{aligned}
\end{equation}
where $\cU^{\bi}$ is defined in \eqref{eq:U_def}.
For the set of points $X$ and corresponding $W(X)$, let $\eta(\bar{x})$ denote the value of
\eqref{model:PLP} when the constraint $x = \bar{x}$ is added to \eqref{model:PLP} for a particular $\bar{x} \in \Omega$.
As we will see below, $\eta$ represents 
the largest lower bound for $f$ induced by the set $X$, and the solution to
\eqref{model:PLP} provides a lower bound for $f$ on $\Omega$.

\begin{lemma}[Underestimator of $f$]\label{lem:underEst}
  If $f$ is convex on $\Omega$, then the optimal value $\eta_*$ of
  \eqref{model:PLP} satisfies $\eta_* \le f(x)$ for all $x \in \Omega$.
\end{lemma}

\begin{proof}
If $W(X)$ is empty, the result holds trivially since $\eta$ is unconstrained. Otherwise,
since \eqref{model:PLP} minimizes $\eta$, it suffices to show
that $\eta(x) \le f(x)$ for arbitrary $x \in \Omega$. Two cases
can occur.
First, if $x \notin \cU^{\bi}$ for every $\bi \in W(X)$, then no conditional cut
exists at $x$. Thus $\eta(x) = -\infty$ and $\eta(x) < f(x)$.
Second, if $x \in \cU^{\bi}$ for some $\bi \in W(X)$, then 
\[ 
(c^{\bi})^\T x + b^{\bi} \le f(x),
\]
by \lemref{cone}, where the poisedness of $X^{\bi}$ follows from the
definition of $W(X)$.
Therefore, $\eta(x) \le
f(x)$ for all $x \in \Omega$.
Since $\eta_* = \ds \min_{x \in \Omega} \eta(x)$,
the result is shown.
\end{proof}
If $W(X)$ in \eqref{model:PLP} is replaced by a proper subset $W'(X) \subset W(X)$ of
multi-indices, then \lemref{underEst} still holds. (This relaxation of
\eqref{model:PLP} associated with removing constraints cannot
increase $\eta_{*}$.) Such a replacement may be necessary if
$W(X)$ becomes too large to allow considering every poised subset of
$X$ when forming \eqref{model:PLP}.

\subsection{Covering $\R^n$ with Conditional Cuts}\label{sec:covering}

Because the cuts in \eqref{model:PLP} are valid only within $\cU^{\bi}$, 
the resulting 
model takes an optimal value of $\eta_*=-\infty$ if there is a point $x \in \Omega$ that
is not in the union of $\cU^{\bi}$ over all $\bi \in W(X)$.
Thus, we find it beneficial to ensure that $X$ contains points that result in a finite
objective value for the underestimator
described by \eqref{model:PLP}. We now establish
a condition that ensures that the union of conditional cuts
induced by $X$ covers $\mathbb{R}^n$ and therefore $\Omega$.

We say that a point $x^0$ belongs to the interior of the convex hull of a set
of points $X = \{ x^1,\ldots,x^{n+1} \}$
if scalars $\alpha_j$ exist such that
\begin{equation}\label{eq:intConvHull}
  x^0 = \sum_{j=1}^{n+1} \alpha_j x^j, \mbox{ where } \sum_{j=1}^{n+1} \alpha_j = 1 \mbox{ and } \alpha_j > 0 \mbox{ for } j = 1,\ldots,n+1.
\end{equation}
This is denoted by $x^0 \in \intp{\convp{X}}$.
\begin{lemma}[Poisedness of Initial Points]\label{lem:intConvPoised}
  If $X = \{x^1,\ldots,x^{n+1}\} \subset \mathbb{R}^n$ is a poised set and if $x^0$ satisfies $x^0 \in \intp{\convp{X}}$, then
all subsets of $n+1$ points in $\{x^0\} \cup X$ are poised.
\end{lemma}
\begin{proof}
  For contradiction, suppose that the set $\{x^0\} \cup X \setminus
  \{x^{n+1}\}$ is not poised and therefore is affinely dependent. Then, 
  there must exist scalars $\beta_j$ not all
  zero, and (without loss of generality) $x^n \in X$ such that
  \begin{equation}\label{eq:lindep}
    \sum_{j=0}^{n-1} \beta_j (x^j - x^n) = 0.
  \end{equation}
  Replacing $x^0$ with \eqref{eq:intConvHull} in the left-hand side above yields 
  \begin{align}\label{eq:sum_2}
    0 &=  \beta_0\left(\sum_{j=1,j\neq n}^{n+1} \alpha_j x^j + (\alpha_n
    -1)x^n\right) + \sum_{j=1}^{n-1} \beta_j (x^j - x^n)  \nonumber\\
    &=  \sum_{j = 1}^{n-1} (\beta_0 \alpha_j + \beta_j) x^j + \bigg(\beta_0(\alpha_n -1)
    -\sum_{j = 1}^{n-1} \beta_j \bigg) x^n + \beta_0 \alpha_{n+1}x^{n+1}.
  \end{align}
  Since $X$ is poised, the vectors $\{ x^1,\ldots, x^{n+1}\}$ are affinely
  independent; by definition of affine independence, the only solution to
  $\sum_{j=1}^{n+1} \gamma_j x^j  = 0$ and $\sum_{j=1}^{n+1} \gamma_j = 0$ is
  $\gamma_j = 0$ for $j = 1,\ldots,n+1$. The sum of the coefficients
  from \eqref{eq:sum_2} satisfies
  \[
  \sum_{j = 1}^{n-1} (\beta_0 \alpha_j + \beta_j) + \beta_0(\alpha_n - 1)
  -\sum_{j=1}^{n-1}\beta_j + \beta_0 \alpha_{n+1} = \beta_0  \sum_{j =
  1}^{n+1} \alpha_j - \beta_0 = 0,
  \]
  because $\sum_{j=1}^{n+1}\alpha_j = 1$. This means that all coefficients
  of $x^j$ in \eqref{eq:sum_2} are also equal to zero.
  Since $\alpha_{n+1}>0$, the last term from \eqref{eq:sum_2} implies 
  that $\beta_0 = 0$.  Considering the remaining coefficients in \eqref{eq:sum_2},
  we conclude that $\beta_0 \alpha_j + \beta_j=0$, which implies
  that $\beta_j = 0$ for $j=1,\ldots, n-1$. This contradicts the assumption
  that not all $\beta_j = 0$.  Hence, the result is proved.
\end{proof}

We now establish a simple set of points that produces 
conditional cuts that cover $\mathbb{R}^n$ and, therefore, the domain $\Omega$. 
\begin{lemma}[Initial Points and Coverage of $\Omega$]\label{lem:startPoints}
Let $X$ be a poised set of $n+1$ points, let $x^0 \in
\intp{\convp{X}}$, and let $W(X \cup \{x^0\})$ be defined as in \eqref{eq:WX_def}.
Then,
\[
\bigcup\limits_{\bi \in W(X\cup\{x^0\})} \cU^{\bi} = \R^n.
\]
\end{lemma}

\begin{proof}
  Since $x^0 \in \intp{\convp{X}}$, there exist $\alpha_j > 0$ such that  
  \begin{equation}\label{eq:spanSet}
    0 = (\sum_{j = 1}^{n+1} \alpha_j) (x^0 - x^0) = (\sum_{j = 1}^{n+1} \alpha_j) x^0 - \sum_{j = 1}^{n+1} \alpha_j x^j = \sum_{j = 1}^{n+1} \alpha_j (x^0-x^j),
  \end{equation}
  where the second equality follows from \eqref{eq:intConvHull}. The existence
  of $\alpha_j>0$ such that $\sum_{j=1}^{n+1}\alpha_j (x^0-x^j) = 0$ implies
  that the vectors $\left\{x^0-x^j: j \in \{ 1,\ldots,n+1\}\right\}$ are a
  positive spanning set (see, e.g., \cite[Theorem~2.3~(iii)]{Conn2009a}). Therefore
  any $x \in \mathbb{R}^n$ can be expressed as 
  \[
  x = \sum_{j=1}^{n} \alpha_j (x^0-x^j),
  \]
  with $\alpha_j \ge 0$ for all $j$.

  We will show that any $x \in \R^n$ belongs to
  $\cU^{\bi}$ for some multi-index $\bi$ containing $x^0$.
  By \lemref{intConvPoised}, every set of $n$ distinct vectors of the form $(x^0 - x^j)$
  for $x^j \in X$ is a 
  linearly independent set. Thus we can express
  \begin{equation}\label{eq:incone}
    x - x^0 = \sum_{j=1,j \neq l}^{n+1} \lambda_j (x^0 - x^j),
  \end{equation}
  for some $l \in \{1, \ldots n+1 \}$. If $\lambda_j \geq 0$ for each
  $j$, then we are done, and $x \in \conep{x^0 - X \setminus \{x^l\}}$.

  Otherwise, choose an index $j'$ such that $\lambda_{j'}$ is the most negative
  coefficient on the right of \eqref{eq:incone} (breaking ties arbitrarily).
  Using \eqref{eq:spanSet}, we can exchange the indices
  $l$ and $j'$ in \eqref{eq:incone} by observing that
  \begin{equation*}
    \lambda_{j'} (x^0 - x^{j'}) =
    \frac{-\lambda_{j'}}{\alpha_{j'}}\left(\sum_{j = 1,j \neq j'}^{n+1} \alpha_j (x^0-x^j)\right).
  \end{equation*}
  Note that $\frac{-\lambda_{j'}}{\alpha_{j'}}\alpha_j>0$ by \eqref{eq:intConvHull},
  and we can rewrite \eqref{eq:incone} as
  \begin{equation}\label{eq:incone1}
    x - x^0 = \sum_{j=1,j \neq j'}^{n+1} \mu_j (x^0 - x^j),
  \end{equation}
  with new coefficients $\mu_j$ that are strictly larger than $\lambda_j$:
  \[ \mu_j = \left\{ \begin{array}{ll} \dps
                      \lambda_j - \frac{\lambda_{j'}}{\alpha_{j'}}\alpha_j > \lambda_j, \quad & j \not = l, j \not = j' \\
                      - \frac{\lambda_{j'}}{\alpha_{j'}}\alpha_j \quad & j = l .
                     \end{array}
             \right. 
  \]
  Observe that \eqref{eq:incone1} has the same form as \eqref{eq:incone} but with
  coefficients $\mu_j$ that are strictly greater than $\lambda_j$. We can now
  define $\lambda\defined\mu$ and repeat the process. If there is some
  $\lambda_{j'}<0$, the process will strictly increase all $\lambda_j$. Because
  there are only a finite number of subsets of size $n$, we must eventually have
  all $\lambda_j \geq 0$. Once
  $\lambda_{j'}$ has been pivoted out, it can reenter only with a positive value (like $\mu_l$ above), so eventually all
  $\lambda_j$ will be nonnegative. 
\end{proof}

\lemref{startPoints} ensures that any poised set of $n+1$ points with an additional
point in their interior will produce conditional cuts that cover $\R^n$. 
\figref{Cuts2D} illustrates this for $n=2$. 
An alternative set of $n+2$ points is
\[ 
X = \left\{ 0, e_1, e_2, \ldots, e_n, -e \right\},
\]
where $e_i$ is the $i$th unit vector and $e$ is the vector of ones.
Larger sets, such as those of the form 
\[ X = \left\{ 0, e_1, -e_1, \ldots, e_n, -e_n \right\}, \]
will similarly guarantee coverage of $\R^n$.

We note that the results in this section do not rely on $X$ or $\Omega$ being a
subset of $\Z^n$. Therefore, the results are readily applicable to the
case when $f$ has continuous and integer variables.

\begin{algorithm2e}[b!]
  \DontPrintSemicolon
  \caption{Identifying a global minimizer of a convex objective on integer $\Omega$.\label{alg:framework}}
  \LinesNumbered
  \SetAlgoNlRelativeSize{-5}
  \SetKwInput{Input}{Input}
  \SetKwInOut{Output}{Output}
  \Input{A set of evaluated points ${X^0} \subseteq \Omega$ satisfying $\left| W(X^0) \right| > 0$ 
  }
  Set $\hat{x} \in \ds \argmin_{x \in X^{0}} f(x)$, upper bound $u_0 \gets f(\hat{x})$, and lower bound $l_0 \gets -\infty$; $k \gets 0$\\
  \While{$l_k < u_k$\label{line:break_out_of_while}} {
    \underline{\textsf{Update:}} Update the piecewise linear program \eqref{model:PLP} using $W(X^k)$ \\
    \underline{\textsf{Lower Bound:}} Solve \eqref{model:PLP} and let its optimal value be $l_{k+1}$ \\
    \underline{\textsf{Next Iterate:}} Select a new trial point $x^{k+1} \in \Omega \setminus X^{k}$ \label{line:select_next_x}\\
    Evaluate $f(x^{k+1})$ and set $X^{k+1}\gets X^{k} \cup \{x^{k+1}\}$ \label{line:adding_to_X}\\
    \eIf{$f(x^{k+1}) < u_k$}
        {
          \underline{\textsf{Upper Bound:}} New incumbent $\hat{x} \gets x^{k+1}$ and upper bound $u_{k+1} \gets f(x^{k+1})$ \label{line:updateF}
        }
        { $u_{k+1} \gets u_{k}$ }
         $k \gets k+1$\\
  }
  \Output{$\hat{x}$, a global minimizer of $f$ on $\Omega$}
\end{algorithm2e}

\section{Convergence Analysis}\label{sec:algm}

We now present \algref{framework} to identify global solutions to
\eqref{eq:MIP-DFO} under \asref{1}. This algorithm constructs a sequence of
underestimators of the form \eqref{model:PLP}.
\secref{PLP}
shows two approaches for modeling the
underestimator and \secref{details} highlights other details that are important for
an efficient implementation of \algref{framework}. For example, the next point evaluated 
can be a solution of \eqref{model:PLP} but this need not be the case.

Note that \eqref{model:PLP} provides a valid lower bound for $f$ on
$\Omega$. If $X \subseteq \Omega$ are points where $f$ has been evaluated,
then $\min\left\{ f(x): x \in X \right\}$ is an upper bound for the minimum of $f$
on $\Omega$. \algref{framework} 
terminates when the upper bound is equal to the lower bound provided by \eqref{model:PLP}. We
observe that \algref{framework} produces a 
nondecreasing sequence of lower bounds provided that
conditional cuts are not removed from \eqref{model:PLP};
we show in \thmref{convergence} that this sequence of lower bounds will
converge to the global minimum of $f$ on $\Omega$.

\algref{framework} resembles a traditional outer-approximation
approach~\cite{bonami.etal:08,Duran1986,Fletcher1994} in that it 
obtains a sequence of lower bounds of \eqref{eq:MIP-DFO} using an underestimator that is
updated after each function evaluation. These function evaluations provide a
nonincreasing sequence of upper bounds on the objective; when the 
upper bound equals the lower bound provided by the underestimator, the
method can terminate with a certificate of optimality.

\algref{framework} leaves open a number of important decisions concerning how  
\eqref{model:PLP} is formulated and solved and how the next iterate is selected.
While we will discuss more involved options for addressing these concerns, 
a simple choice would be to add all new possible cuts and let the next iterate be a
minimizer of \eqref{model:PLP}.
If this minimizer is not chosen, a (possibly difficult) separation
problem  may have to be solved to obtain a new iterate in
\linerefa{select_next_x}, for example, when $\Omega$ is sparse.
Although such choices can result in computational difficulties, these choices are
useful for showing the behavior of
\algref{framework}, which we do now. In \figref{1dExCP} we see three
iterations of \algref{framework} solving
the one-dimensional problem
\[ \mini \; f(x)=x^2 \quad \st \; x \in [-4,4], x\in\mathbb{Z}.\]
Black dots indicate interpolation points where $f$ has been previously evaluated, and green dots
indicate the solution to \eqref{model:PLP} in each iteration. The solid red lines show
the piecewise linear underestimator of the function. We 
observe that \lemref{cone} can be strengthened for
one-dimensional problems where conditional cuts underestimate convex 
$f$ at all points outside the convex hull of the points used to
determine the corresponding secant function. (This is not true for $n>1$.)

\begin{figure}[t]
\center
\captionsetup[subfloat]{labelformat=empty}
\subfloat[Iter.~1: $l_k=-9,u_k=1,\hat{x}=1$]{\includegraphics[width=0.33\textwidth]{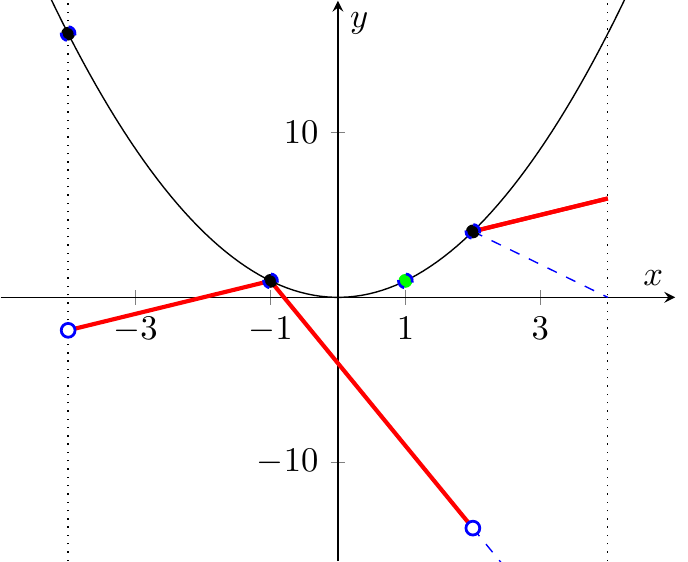}}
\subfloat[Iter.~2: $l_k=-0.27,u_k=0,\hat{x}=0$]{\includegraphics[width=0.33\textwidth]{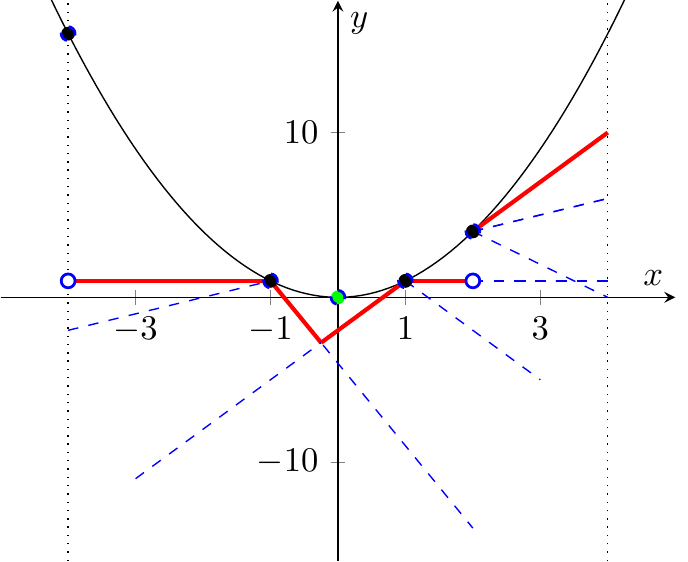}}
\subfloat[Iter.~3: $l_k=0,u_k=0,\hat{x}=0$]{\includegraphics[width=0.33\textwidth]{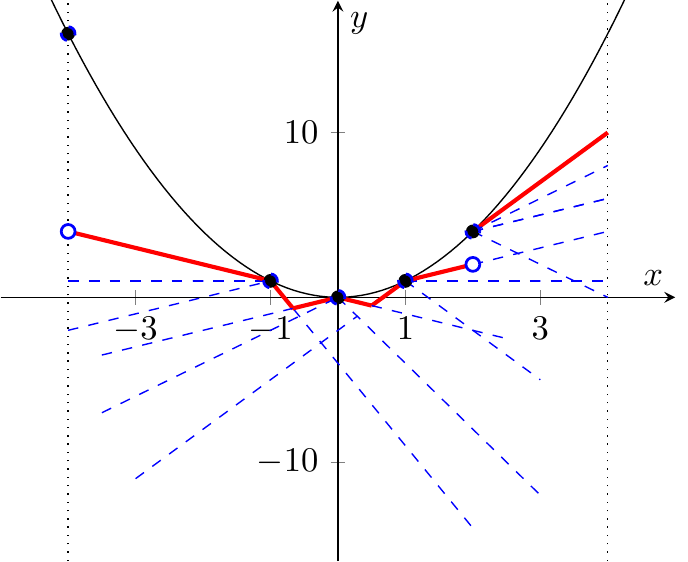}}
\caption{Illustration of \algref{framework} minimizing $f(x)=x^2$ on $[-4,4] \cap \Z$.\label{fig:1dExCP}}
\end{figure}

We now prove that \algref{framework} identifies a global minimizer
of $f$.
\begin{theorem}[Convergence of Algorithm~\ref{alg:framework}]\label{thm:convergence}
If \asref{1} holds, \algref{framework} terminates at an
optimal solution $x^*$ of \eqref{eq:MIP-DFO} in finitely many iterations.
\end{theorem}
\begin{proof}
  \algref{framework} will terminate in a finite number of iterations
  because \asref{1} ensures that $\Omega$ is finite and \linerefa{select_next_x}
  ensures that $x^k$ is not a previously evaluated element of $\Omega$.

  For contradiction, assume that \algref{framework} terminates at iteration $k'$
  with $f(\hat x) > f(x^*)$ for some $x^* \ds \in \argmin_{x \in \Omega} f(x)$.
  It follows from \linerefa{updateF} that $x^* \notin X^{k'}$, because 
  $f(x^*)<f(\hat x)$.
  \lemref{underEst} ensures that the value of each conditional cut
  at $x^*$ is not larger than $f(x^*)$, which implies
  that $\eta(x^*)\leq f(x^*)$. Thus, the lower bound satisfies
  \[ l_{k'} \leq f(x^*) < f(\hat x) = u_{k'} . \]
  Since $l_{k'} < u_{k'}$, \algref{framework} did not terminate at iteration $k'$,
  giving a contradiction. 
  Therefore, the result is shown.
\end{proof}

A special case of \thmref{convergence} ensures that \algref{framework} terminates with
a global solution of \eqref{eq:MIP-DFO} when $x^k$ is an optimal solution of
\eqref{model:PLP}.
As in most integer optimization algorithms, the termination
condition $l_{k}=u_{k}$ may be met before $X^{k} = \Omega$; that is, \algref{framework}
need not evaluate all the points in $\Omega$. Termination before $X^k = \Omega$ occurs 
when \eqref{model:PLP} is refined in Step 4 of \algref{framework} and the lower
bound at all points (and on the optimal value of \eqref{eq:MIP-DFO}) is
tightened. 
When the next iterate is chosen, the upper bound typically improves
and convergence occurs faster than enumeration.
Yet, in the worst-case scenario, when
cuts are unable to refine \eqref{model:PLP}, the algorithm will evaluate all
of $\Omega$. The lower bound at each point in $\Omega$ will be equal to the
function value at that point, which will imply $l_{k}=u_{k}=\min_{x \in \Omega} f(x)$ for
\eqref{model:PLP} (ensuring termination of the algorithm).

\section{Implementation Details}\label{sec:implementing} 
\algref{framework} relies critically on the underestimator described by
\eqref{model:PLP}. \secref{PLP} develops two approaches for formulating
\eqref{model:PLP}, and \secref{details} discusses important
details for efficiently implementing \algref{framework}. \secref{shucil}
combines these details in a 
description of our preferred method for solving \eqref{eq:MIP-DFO}, 
\code{SUCIL}.

\subsection{Formulating \eqref{model:PLP}}\label{sec:PLP}
We present two methods for encoding 
\eqref{model:PLP} and thereby obtain lower bounds for \eqref{eq:MIP-DFO}. The
first approach formulates \eqref{model:PLP} as a mixed-integer linear program
(MILP) using binary variables to indicate when a point $x$ is in $\cU^{\bi}$ for some multi-index $\bi$. 
Unfortunately, the resulting MILP
is difficult to solve for even small problem instances. This motivates the
development of the second approach, which
directly builds an enumerative model of \eqref{model:PLP} in the
space of the original variables only.

\subsubsection{MILP Approach}\label{sec:tols} 
Formulating \eqref{model:PLP} as an MILP requires forming the
secant function 
$m^{\bi}(x) = (c^{\bi})^\T x + b^{\bi}$ 
corresponding to each multi-index $\bi \in W(X)$. Since $m^{\bi}$ is 
valid only in $\cU^{\bi}$ (see \lemref{underEst}), we use binary variables to encode when $x \in
\cU^{\bi}$. Explicitly, for each $\bi \in W(X)$ and each $i_j \in \bi$, our MILP model 
sets the binary
variable $z^{{i}_j}$ to be 1 if and only if $x \in \conep{x^{i_j} - X^{\bi}}$.
Although the forward implication can be easily
modeled by using continuous variables
$\lambda^{i_j}$, we must introduce additional binary
variables $w^{i_j}$ for the reverse implication.

We now describe the
constraints in the MILP model. The first set of constraints ensures that
$\eta$ is no smaller than any of the
conditional cuts that underestimate $f$:
\begin{equation}\label{eq:cuts}
\eta \geq (c^{\bi})^\T x + b^{\bi} - M_{\eta} \left(1-\sum_{j=1}^{n+1}z^{i_j}\right),
\quad \forall \bi \in W(X),
\end{equation}
where $M_{\eta}$ is a sufficiently large constant.
By \lemref{only_one_cone}, we can add constraints to ensure that $x \in \Omega$
belongs to no more than one of the cones in $\cU^{\bi}$ for a given $\bi$:
\begin{equation}\label{eq:SOScone}
  \sum_{j = 1}^{n+1}z^{i_j} \le 1, \; \forall \bi \in W(X).
\end{equation}
The following constraints define each point $x \in \Omega$ as a linear
combination of the extreme rays of each $\conep{x^{i_j}-X^{\bi}}$:
\begin{equation}\label{eq:xInCone}
  x = x^{i_j} + \sum_{l=1, l \neq j}^{n+1}
     \lambda^{i_j}_l \left(x^{i_j} - x^{i_l}\right),
     \quad \forall \bi \in W(X),
     \;\forall i_j \in \bi.
\end{equation}
To indicate that $x
\in \conep{x^{i_j}-X^{\bi}}$, 
the following constraints enforce a lower bound of 0 on
$\lambda$ when the corresponding $z^{i_j}=1$:
\begin{equation}\label{eq:lambdaLb}
  \lambda^{i_j}_l \geq -M_\lambda \left(1-z^{i_j}\right),
    \quad \forall \bi \in W(X),
    \; \forall i_j,i_l \in \bi, j \neq l,
\end{equation}
where $M_{\lambda}$ is a sufficiently large constant. 
Next, we introduce the binary variables $w_l^{i_j}$ that are 1 when the corresponding
variable $\lambda_l^{i_j}$ is nonnegative.
The following constraints model the condition that $w_l^{i_j}=0$ implies that the corresponding
$\lambda_l^{i_j}$ takes a negative value:
\begin{equation}\label{eq:lambdaUb}
  \lambda^{i_j}_l \leq -\epsilon_\lambda + M_\lambda w^{i_j}_l,
    \quad \forall \bi \in W(X),
    \; \forall i_j,i_l \in \bi, j \neq l,
\end{equation}
where $\epsilon_\lambda$ is a sufficiently small positive constant.
The last set of constraints force at least one of the $w$ variables to be
0 if the corresponding $z$ is 0:
\begin{equation}\label{eq:zw}
  nz^{i_j} \leq  \sum_{l=1, l \neq j}^{n+1} w^{i_j}_l \leq n - 1 + z^{i_j},
     \quad \forall \bi \in W(X),
     \;\forall i_j \in \bi.
\end{equation}
The full MILP model encoding of \eqref{model:PLP} is
\begin{equation} \label{model:CPFull} \tag{CPF}
  \begin{aligned}
\dps \mini_{x,\lambda,z,w} \; & \eta \\
\st  \; & \dps \mbox{\eqref{eq:cuts}--\eqref{eq:zw}}\\
     \; & \dps w^{i_j}_l, z^{i_j} \in \{0,1\}, \quad \forall l,j \in
     \{1,\ldots, n+1\}, l \neq j; \quad \forall \bi \in W(X) \\
     \; & \dps x \in \Omega .
  \end{aligned}
\end{equation}

The constants $M_{\eta}, M_{\lambda}$, and $\epsilon_{\lambda}$ must be chosen
carefully in order to avoid numerical issues when solving \eqref{model:CPFull}; see \appref{tolerance}
for more details. In early numerical results, we
observed that taking large values for
$M_{\eta}$ and $M_{\lambda}$ and small values for $\epsilon_{\lambda}$ resulted in
numerical issues for the MILP solvers.
In an attempt to remedy
this situation, we derived cuts (also described in
\appref{tolerance}) in which $c^{\bi}$ and $b^{\bi}$ are integer valued.
One can then show, for example, that $1/\|c^{\bi}\|_2$
is a valid lower bound for $\epsilon_{\lambda}$, and similar tight bounds can be derived for
$M_{\lambda}$. With these tighter constants, some numerical issues were
resolved. Yet, the growth in the number of constraints in \eqref{model:CPFull}
prevented its application to problems with $n\geq3$.
\begin{figure}
\center
\subfloat{\includegraphics[height=2.1in]{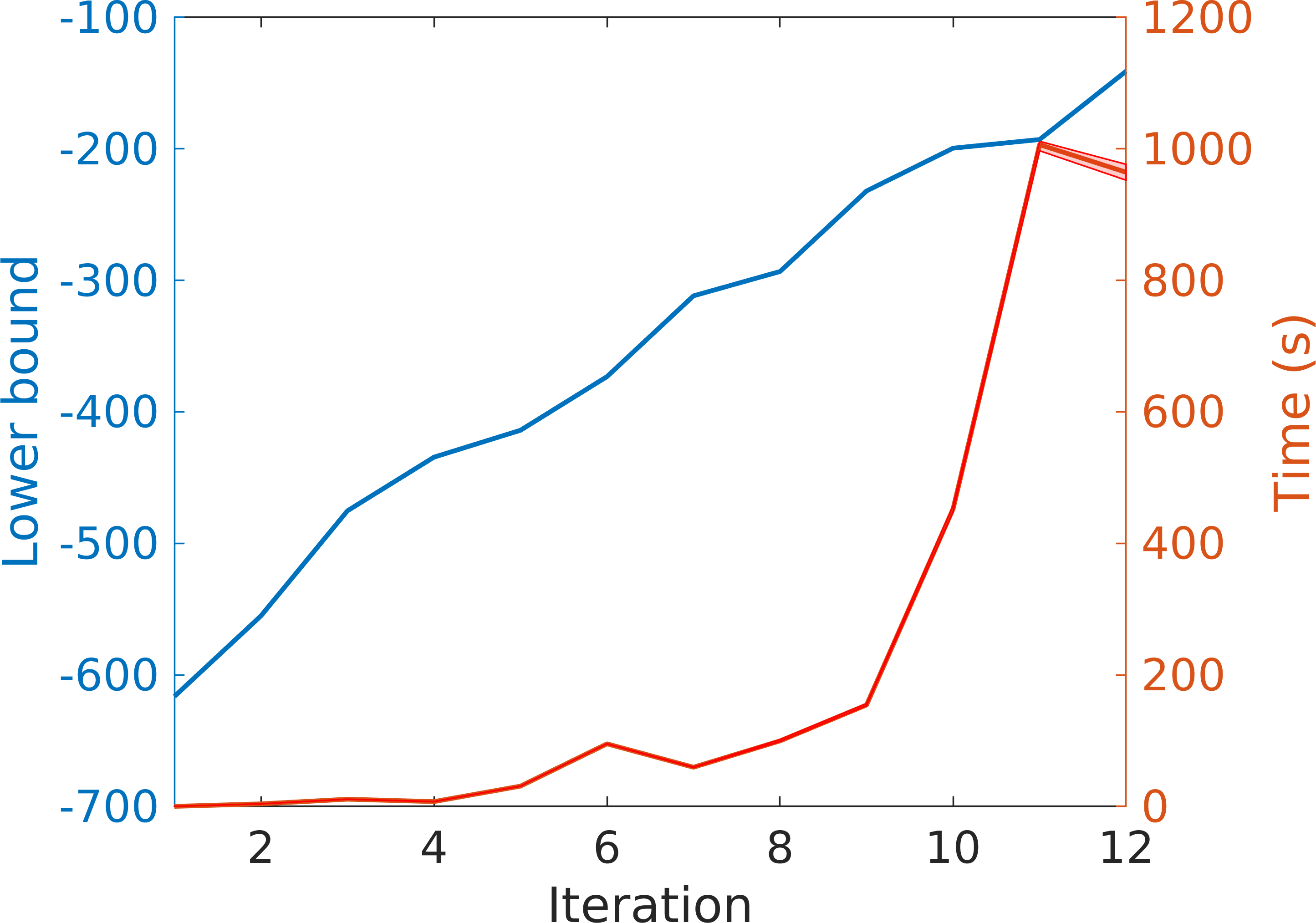}} 
\hfill
\subfloat{\includegraphics[height=2.1in]{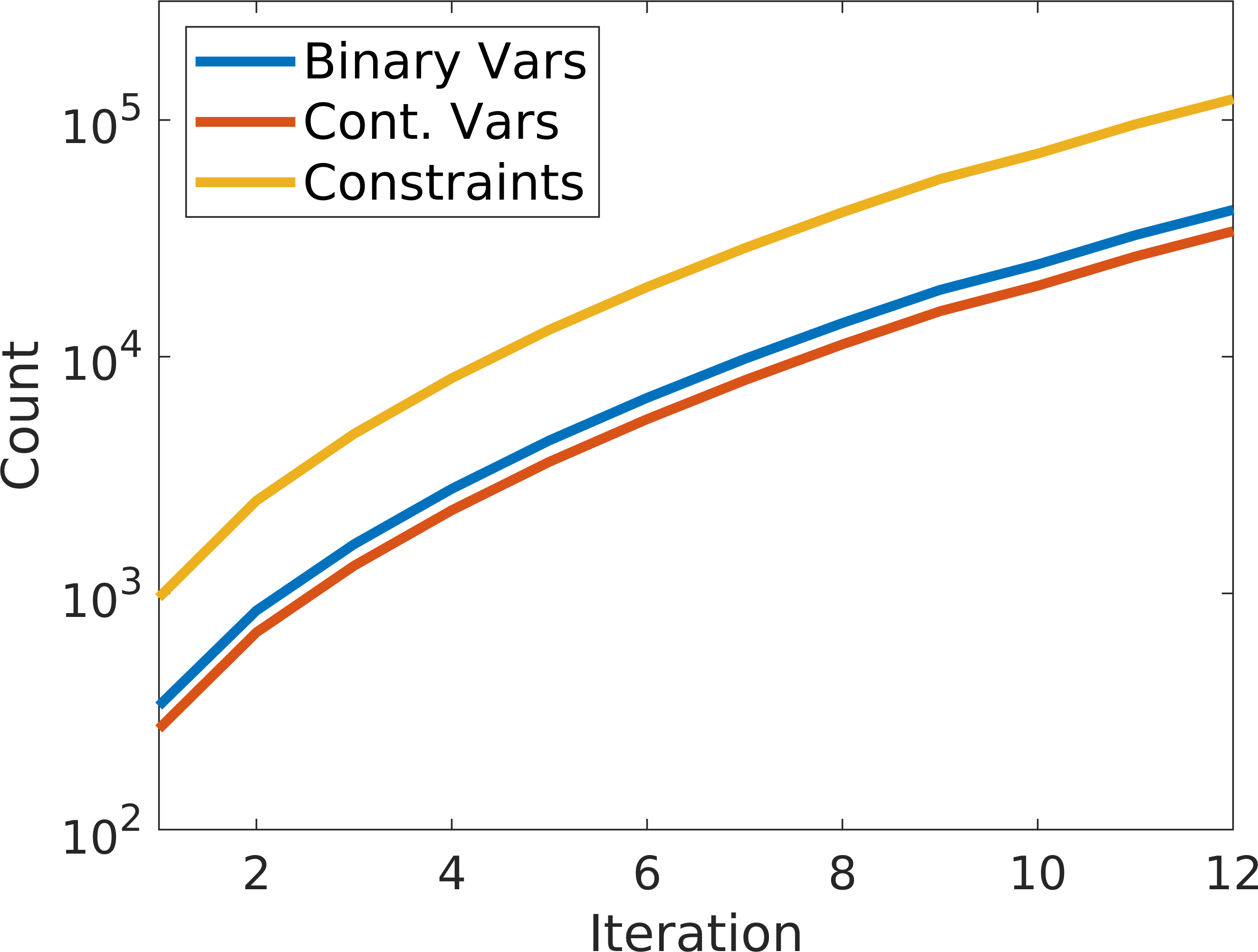}} 
\caption{Characteristics of the first 12 instances of \eqref{model:CPFull}
generated by \algref{framework} minimizing the convex quadratic 
\prob{abhi} on $\Omega=[-2,2]^3 \cap \Z^3$. Left shows the lower bound and solution
time (mean of five replications and maximum and minimum times are also shown); right shows the number of binary and 
continuous variables and constraints.
For further details of these 12 MILP models, see \tabref{mipExpo} in
\appref{mip}.
\label{fig:MIPExpo}}
\end{figure}

Initial versions of the MILP model \eqref{model:CPFull} resulted in large
times to solution.
\figref{MIPExpo} shows the behavior of
\algref{framework}---when adding all possible cuts when updating
\eqref{model:PLP} and choosing the next iterate be a minimizer of
\eqref{model:PLP}---when minimizing the convex quadratic function \prob{abhi}
(defined in \tabref{probs} of \appref{probs}) on $\Omega=[-2,2]^3 \cap \Z^3$.
We note that the variations in CPU time are consistent over five repeated
runs and vary by less than 2.4\% for the last two iterations.
The solution time of MILP solvers depends critically on implementation
features, including presolve operations, 
node selection rules, and branching preferences. 
After the additional set of cuts
(constraints) are introduced in iteration 12 of this problem instance, the MILP solver
was able to solve the problem in slightly less time than the
previous iteration. (Such occurrences are not rare in
MILPs: time to solution is not strictly increasing in problem size.)
Overall, we find that the growth in CPU time is due to the increasing number of
conditional cuts and the associated explosion in the number of binary and
continuous variables. This trend appears to limit the applicability of the MILP
approach. Note that the global minimum of \prob{abhi} on $[-2,2]^3 \cap \Z^3$
has not yet been encountered when the MILPs become too large to solve. (The
iteration 13 MILP was not solved in 30 minutes.) 

\subsubsection{Enumerative Approach}\label{sec:enum}

Whereas the MILP from \secref{tols} encodes information about 
every conditional cut in a single model,  
this section considers an alternative approach of updating the value of
$\eta(x)$ for each $x \in \Omega$ as new conditional cuts are encountered.
After the information from a new secant function is used to update
$\eta(x)$, the secant is discarded. 

Ordering the finite set of feasible integer points as
$\left\{ x^{1}, x^{2}, \ldots, x^{\left| \Omega \right|} \right\}$, 
our approach maintains and updates a vector of bounds 
\begin{equation} \label{eq:vec_of_bounds}
\left[ \eta(x^1), \eta(x^2), \ldots, \eta(x^{\left| \Omega \right|})
\right]^\T \in \R^{\left|
\Omega \right|},
\end{equation}
where $\eta(x^j)$ is the value of \eqref{model:PLP} when $x = x^j$.
The value of $\eta(x^j)$ is initialized to $-\infty$; and as each secant
is constructed, $\eta(x^j)$ is set to the maximum of its current value and
the value of the conditional cut at $x^{j}$. This procedure is 
described in \algref{update}.
Since the important information about each conditional cut will be stored in
$\eta(x)$, the secants defining each cut do not need to be stored. 
Furthermore, if $\eta_k(x)$ is the value of the underestimator \eqref{eq:vec_of_bounds} at iteration $k$, 
then solving each instance of \eqref{model:PLP}
corresponds to looking up $\ds \argmin_{j \in \{ 1,\ldots,\left| \Omega \right| \}}
\eta_k(x^j)$ (breaking ties arbitrarily). 
Similarly, termination of \algref{framework} requires testing only that 
$\ds \min_{ j \in \{ 1,\ldots,\left| \Omega \right|\}} \; \eta_k(x^j) \geq u_k$.

Note that when solving \eqref{eq:MIP-DFO}, updating
$\eta(x)$ for all $x \in \Omega$ is unnecessary. Rather, one needs to
update $\eta(x)$ only at points that could possibly be a global minimum of $f$
on $\Omega$. When $f$ is evaluated at $x^{k+1}$ and a multi-index $\bi \in
W(X^k \cup x^{k+1})$ is encountered that is not in $W(X^k)$, we update the
lower bound only at points in $\cU^{\bi}$ that are also in
\begin{equation}\label{eq:Omega_k_def}
  \Omega_k \defined \{ x \in \Omega \setminus X^k : \eta_k(x) < u_k \}.
\end{equation}
That is, we update $\eta_k(x)$ for points in $\cU_k^{\bi} = \Omega_k \cap
\cU^{\bi}$ for each newly encountered $\bi$. 

\begin{algorithm2e}
  \DontPrintSemicolon
\caption{Routine for updating lower bound for $f$ at each point in $\Omega$.\label{alg:update}}
  \SetAlgoNlRelativeSize{-5}
  \SetKwProg{Pn}{Function}{:}{return}
  \SetKwFunction{FunUp}{UpdateEta}
  \Pn{\FunUp{$X^{\bi}$,$b^{\bi}$,$c^{\bi}$,$\cU^{\bi}_k$,$\eta(x)$}}{
    \For{$i_k \in \bi$} {
      \For{$j=1,\ldots,\left| \Omega \right|$} {
        \If {$x^{j} \in \conep{x^{i_k}-X^{\bi}} \cap \cU^{\bi}_k$} {
          $\eta(x^j) \gets \max \left( \eta(x^j), (c^{\bi})^\T x^{j} + b^{\bi} \right)$
        }
      }
    }
  }
\end{algorithm2e}

\subsection{Other Implementation Details}\label{sec:details}
The enumerative approach of maintaining the value of the underestimator
$\eta(x)$ described in \secref{enum} avoids many of the computational
pitfalls of the MILP model discussed in \secref{tols}. Below, we discuss
additional computational enhancements that lead to an efficient implementation of 
\algref{framework} in conjunction with \algref{update}.

\subsubsection{Checking Whether $X^{\bi}$ Is Poised and Whether $x \in \cU^{\bi}$}\label{sec:checking_poised}
We now describe a numerically efficient representation of
$\conep{x^{i_j} - X^{\bi}}$ for $i_j \in \bi$. 
Given a poised set of $n+1$ points, $X^{\bi}$, for
each $i_j \in \bi$ we define a secant function satisfying 
\begin{align}
  (c^{i_j})^\T x^{i_l} + b^{i_j} &= 0, \text{ for all } i_l \in \bi, i_l \neq i_j, \mbox{ and }\label{eq:hyp1}\\
  (c^{i_j})^\T x^{i_j} + b^{i_j} &> 0.\label{eq:hyp2}
\end{align}
Only one such secant exists for each $i_j \in \bi$; however, the
representation of this secant is not unique since
$(c^{i_j},b^{i_j})$ are obtained by solving an underdetermined system of equations.
Given $(c^{i_j},b^{i_j})$ satisfying \eqref{eq:hyp1} and \eqref{eq:hyp2},
we define the corresponding halfspace,
\begin{equation}\label{eq:hypDef}
H^{i_j} \defined \{x:(c^{i_j})^\T x + b^{i_j} \leq 0\}.
\end{equation}
We now show that $\conep{x^j - X^{\bi}}$, defined in \eqref{eq:cone_def}, can
be represented as the intersection of $n$ such halfspaces.

\begin{lemma}[Set Equality]\label{lem:coneFacet}
  For a poised set $X^{\bi}$, $\conep{x^{i_j} -
  X^{\bi}} = F^{i_j} \defined \bigcap\limits_{i_l \neq i_j} H^{i_l}$ 
  for each $i_j \in \bi$.
\end{lemma}

\begin{proof}
  Let $\bi$ be given and $i_j \in \bi$ fixed.
  We first show that $\conep{x^{i_j} - X^{\bi}} \subseteq F^{i_j}$ by showing
  that an arbitrary $x \in \conep{x^{i_j} - X^{\bi}}$ satisfies
  \eqref{eq:hypDef} for each $i_l \in \bi, i_l \neq i_j$.
  Given $(c^{i_l},b^{i_l})$
  satisfying \eqref{eq:hyp1} and \eqref{eq:hyp2}, then using the definition \eqref{eq:cone_def} yields
  \begin{align*}
    (c^{i_l})^\T x + b^{i_l}
    &= (c^{i_l})^\T \left(x^{i_j} + \sum_{k=1, k\neq j}^{n+1} \lambda_k (x^{i_j} - x^{i_k})\right) +  b^{i_l}\\
    &= (c^{i_l})^\T x^{i_j} +  b^{i_l} + \sum_{k=1, k\neq j}^{n+1} \lambda_k (c^{i_l})^\T x^{i_j}
    - \sum_{k=1, k\neq j}^{n+1} \lambda_k (c^{i_l})^\T  x^{i_k} \\
    &= 0 + \sum_{k=1, k\neq j}^{n+1} \lambda_k \left ( (c^{i_l})^\T x^{i_j} + b^{i_l} \right )
    - \sum_{k=1, k\neq j}^{n+1} \lambda_k \left ( (c^{i_l})^\T  x^{i_k} + b^{i_l} \right )\\
    &= 0 - \sum_{k=1, k\neq j}^{l-1} \lambda_k \left ( (c^{i_l})^\T  x^{i_k} + b^{i_l} \right )
    - \lambda_l \left( (c^{i_l})^\T x^{i_l} +  b^{i_l} \right)
    - \sum_{k=l+1, k\neq j}^{n+1} \lambda_k \left ( (c^{i_l})^\T  x^{i_k} + b^{i_l} \right )\\
    &= - \lambda_l \left( (c^{i_l})^\T x^{i_l} +  b^{i_l} \right) \le 0,
  \end{align*}
  where we have used \eqref{eq:hyp1} in the last three equations.
  The final inequality holds because $\lambda_l \geq 0$ by \eqref{eq:cone_def}
  and $(c^{i_l})^\T x^{i_l} +  b^{i_l} > 0$ by \eqref{eq:hyp2}. Because $i_l$ is
  arbitrary, it follows that any $x$ in $\conep{x^{i_j} - X}$ is also in $F^{i_j}$.

  We now show that $F^{i_j} \subseteq \conep{x^{j} - X^{\bi}}$ by contradiction.
  If $x \notin \conep{x^{i_j} - X^{\bi}}$ for a set of $n+1$ poised points
  $X^{\bi}$, then $x$ can be represented as 
  $x^{i_j} +\ds \sum_{l=1, l\neq j}^{n+1} \lambda_l \left( x^{i_j} - x^{i_l}
  \right)$ only with some $\lambda_{l} < 0$. Thus, \eqref{eq:hypDef} is
  violated for some $l$, and hence $x \notin F^{i_j}$.
\end{proof}

\lemref{coneFacet} gives a representation of each $\conep{x^{i_j} -
X^{\bi}}$ involving $n$ halfspaces that differs from $\conep{x^{i_l} -
X^{\bi}}$ for $i_l\in \bi$, $i_l \neq i_j$ in only one component.
Therefore, we can represent $\cU^{\bi}$ via only $n+1$ halfspaces.
We efficiently calculate these halfspaces by utilizing the
QR factorization 
$ 
  \left[ Q^{\bi} \; R^{\bi} \right] =
  \left[ \bar{X}^{\bi} \; e \right]^\T$.
If $R^{\bi}$ has positive diagonal entries, then the multi-index
$\bi$ corresponds to a poised set $X^{\bi}$. 
The coefficients in each $(c^{i_j},b^{i_j})$ can be obtained by
updating $Q^{\bi}$, $R^{\bi}$ by deleting the corresponding column from $\left[
\bar{X}^{\bi} \; e \right]^\T$.
The sign of $(c^{i_j},b^{i_j})$ can be changed in order to ensure that
\eqref{eq:hyp2} holds.

\subsubsection{Approximating $W(X^k \cup \{x^{k+1}\})$}
The use of $\eta_k(x)$ to store the lower bound at each $x \in \Omega_k$ allows
us to avoid encoding all secants in $W(X^k)$. After $f$ has been evaluated
at a new point $x^{k+1}$, constructing the tightest possible underestimator in
$\eta_k$ requires considering multi-indices $\bi$ in $W(X^k \cup \{x^{k+1}\})$ that
contain $x^{k+1}$. (Combinations not containing $x^{k+1}$ have already been
considered in previous iterations.)
While not storing secants is significantly more computationally efficient
than encoding and storing all secants in $W(X^k)$, it still results in
checking the poisedness of prohibitively many sets of $n+1$ points. For
example, if $\left| X^k \right| = 100$ and $n=5$, over 75 million
QR factorizations must be performed, as discussed in \secref{checking_poised}. 

Therefore, as an alternative, we seek a small, representative subset of multi-indices of
$W(X^{k})$ by identifying a subset of points that will yield the best
conditional cuts.
\begin{definition}\label{def:generators}
  Let $\bar{W}_k$ be the set of multi-indices in $W(X^k)$ that define the
  largest lower bound at some point in $\Omega_k$ (defined in
  \eqref{eq:Omega_k_def}). That is,
  $\bar{W}_k \defined \{ \bi : \exists x \in \Omega_k \text{ such that } \eta_k(x) = m^{\bi}(x) \}$.
  We denote to the \emph{generator} set of points as
  $G^k \defined \{x^j : \exists \bi \in \bar{W}_k \text{ such that } j \in \bi
  \}$.
\end{definition}
Hence, $G^k$ contains points that define $\eta_k(x)$ for at least one $x \in
\Omega_k$. Using $W(G^k)$ in place of $W(X^k)$ does relax \eqref{model:PLP}, yet
the lower bounding property of \eqref{model:PLP} still remains. We show
below that this change does not
affect the finite termination property of \algref{framework} provided at least one
cut is added for every new $x^{k+1}$. 

\begin{figure}
  \begin{center}
    \includegraphics[width=0.45\linewidth]{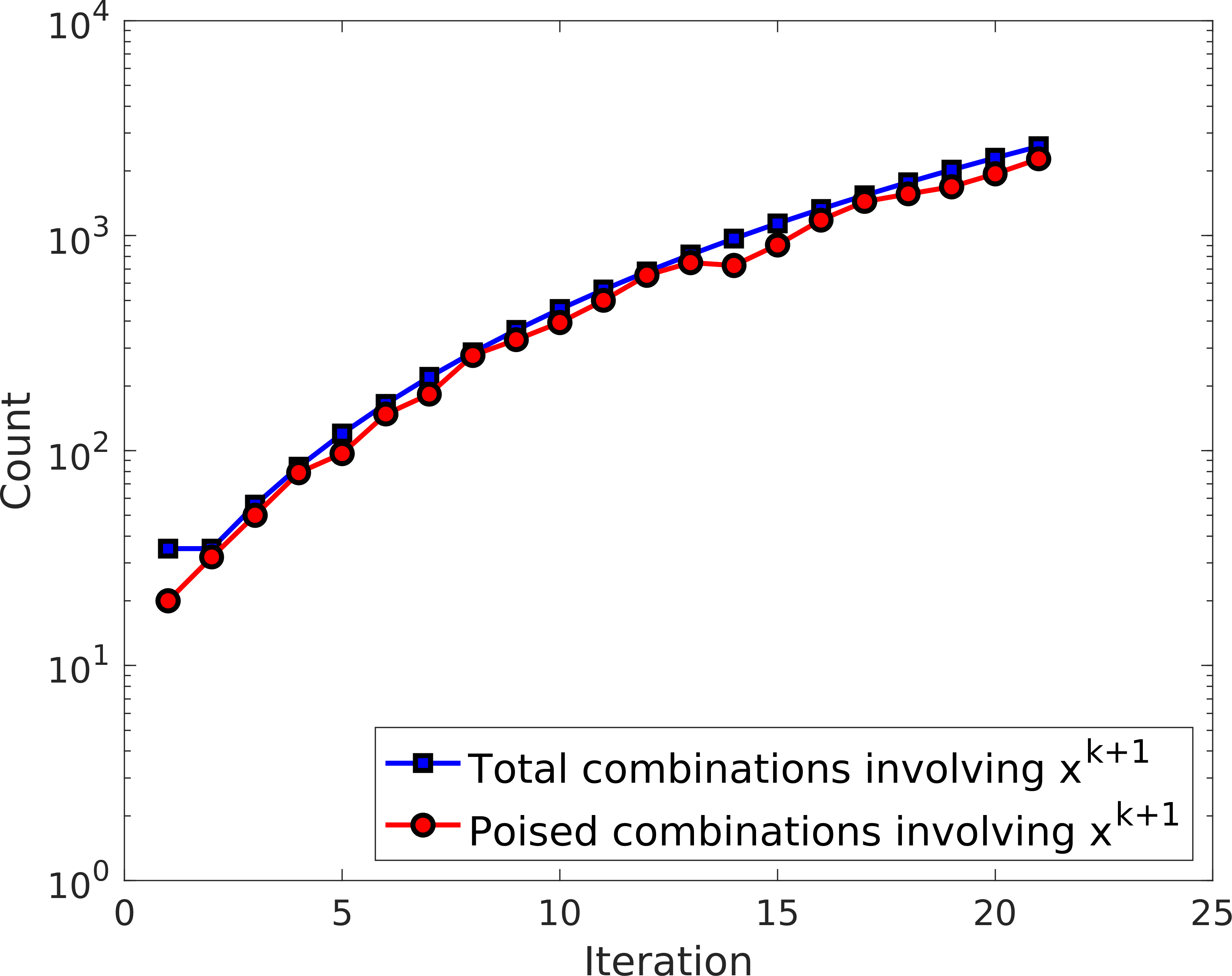}
    \includegraphics[width=0.45\linewidth]{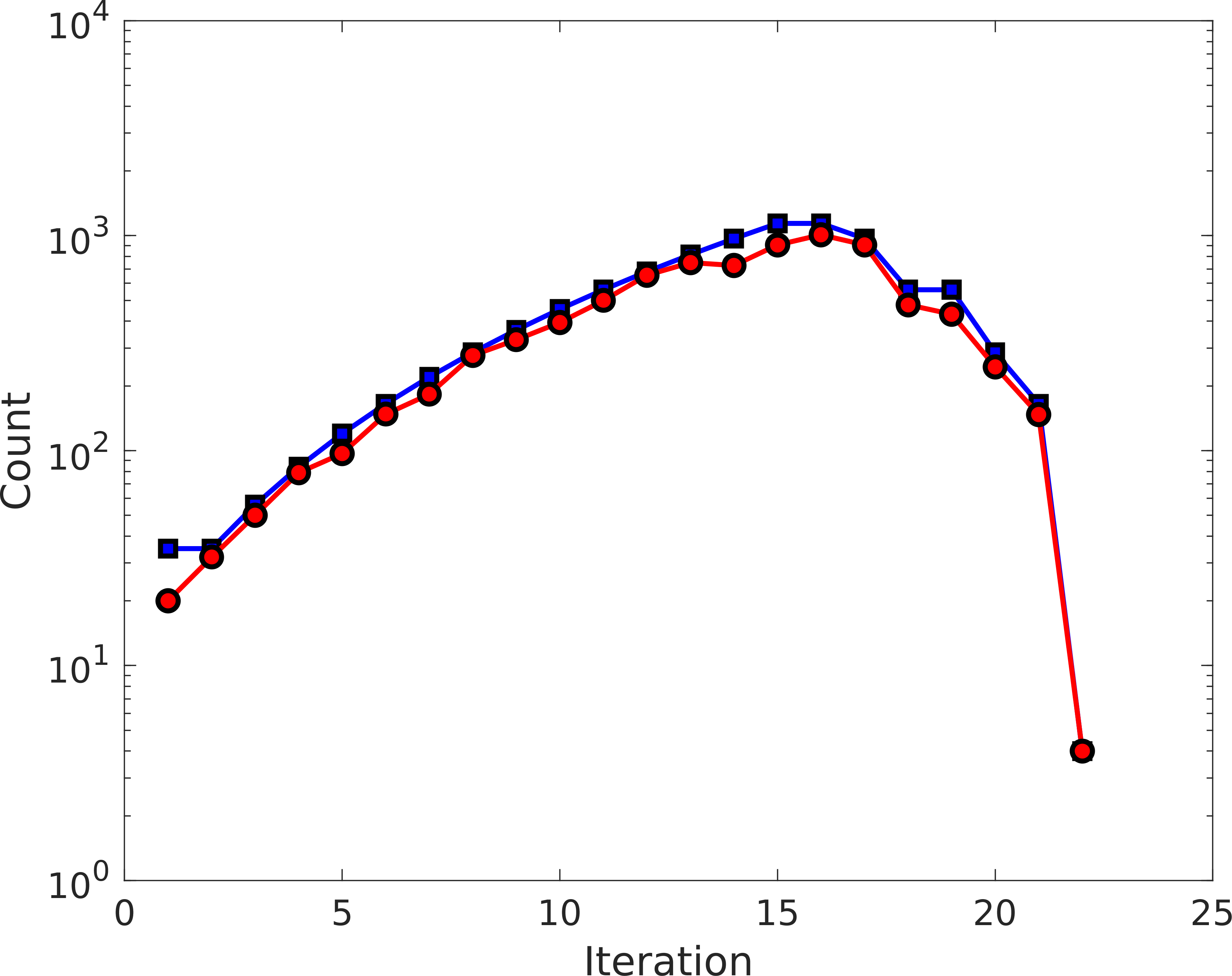}
    \caption{Number of total combinations and poised combinations that include 
    $x^{k+1}$ in $W(X^k)$ (left) and $W(G^k)$ (right) when minimizing
    \prob{quad} on $\Omega = [-4,4]^3 \cap \Z^3$.\label{fig:allVsGen}}
  \end{center}
\end{figure}

\figref{allVsGen} compares the growth of the number of subsets of indices that
must be considered when determining whether a multi-index $\bi$ is poised or
not when using \algref{framework} to minimize \prob{quad} (\tabref{probs} in \appref{probs}) on 
$\Omega = [-4,4]^3 \cap \Z^3$.
Preliminary numerical experiments showed that although a high percentage of
all combinations in $W(X^k \cup x^{k+1})$, which involve the new iterate
$x^{k+1}$ at an iteration $k$, are poised, only a small fraction of these
actually update the lower bound at any point in $\Omega_k$ (we elaborate more
on this in \secref{chall}).

\subsubsection{Selecting $x^{k+1}$}
Early experiments with our algorithm showed that it 
spent 
many early
iterations evaluating points at the boundary of $\Omega$. 
Although \secref{covering} provides a method for ensuring that all $x \in \Omega$
are bounded by at least one conditional cut, the solution to \eqref{model:PLP}
is often at the boundary of $\Omega$. Rather than moving so far from a
candidate solution,
we consider a trust-region approach to keep iterates close
to the current incumbent. 
As long as we maintain a lower bound for $f$ on $\Omega$, the convergence
proof in \thmref{convergence} does not depend on $x^{k+1}$ being the global
minimizer of our lower bound.

In practice, we use an infinity-norm trust region and set the minimum
trust-region radius, $\Delta_{\min}$ to 1. At iteration $k$, the maximum
radius that must be considered is 
$\max_{x,y \in \Omega_k, x \neq y} \left\| x - y \right\|_\infty$.

\subsection{The \code{SUCIL} Method}\label{sec:shucil}
We now present the \code{SUCIL} 
method for obtaining global solutions to
\eqref{eq:MIP-DFO} under \asref{1}.
The algorithm using the trust-region step is shown in \algref{lookup}.
We observe that \algref{lookup}
maintains a valid lower bound $\eta_k(x)$ at every point, $x \in \Omega_k$, and
that the trust-region mechanism ensures that the algorithm terminates only when 
the lower bound equals the best observed function value.

\begin{algorithm2e}[h!]
    \DontPrintSemicolon
    \caption{\code{SUCIL}: secant underestimator of convex functions on the integer lattice.\label{alg:lookup}} 
    \SetAlgoNlRelativeSize{-5}
    \SetKw{true}{true}
    \SetKw{break}{break}
    \SetKwInput{Input}{Input}
    \SetKwInOut{Output}{Output}
    \SetKwFunction{FunUp}{UpdateEta}
    \SetKwFunction{FunTR}{NewIterateTR}
    \Input{A set of evaluated points ${X^0} \subseteq \Omega$ : $\bigcup\limits_{\bi \in W(X^0)} \cU^{\bi} = \R^n$ and trust-region radius lower bound $\Delta_{\min}\geq 1$}
    Set $\hat{x} \in \dps \argmin_{x \in X^{0}} f(x)$,
    upper bound $\ds u_0 \gets f(\hat{x})$, 
    $\Omega_0 \gets \Omega$,
    and $k\gets 0$\\ 
    Initialize lower bounding function $\eta_{-1}(x)\gets -\infty$ for all $x\in \Omega$; set lower bound $l_0 \gets -\infty$ \\
    \While{$l_k < u_k$}
          {
            \underline{\textsf{Update:}} \\
            Generate $G^k$ (according to \defref{generators}) using $X^k$\\
            \For{$\bi \in W(G^k)$}
                {
                  Compute $QR$ factors: $[Q,R] \gets \operatorname{qr}([e\; X^{\bi}])$\\
                  \If{$X^{\bi}$ is poised}
	             {
	               Find coefficients $c^{\bi}, b^{\bi}$ and form set $\cU^{\bi}_k \gets \Omega_k \cap \cU^{\bi}$ using $QR$ factors\\
                       Update look-up: $\eta_k \gets $ \FunUp{$X^{\bi}$,$b^{\bi}$,$c^{\bi}$,$\cU^{\bi}_k$,$\eta_{k-1}$}; see \algref{update}
	             }
                }
            \underline{\textsf{Lower Bound:}} \\
            $l_{k+1} \gets \dps \min_{x \in \Omega_k}{\eta_k(x)}$ from look-up table\\
	    \If{$l_{k+1}=u_k$}{\break}
            \underline{\textsf{Next Iterate:}} \\
            Update $\Omega_{k} \gets \{ x \in \Omega \setminus X^k : \eta_k(x) < u_k \}$
            \label{line:select_next_x1}\\
        \eIf{$\left\{x \in \Omega_k: \| x - \hat{x} \| \leq \Delta_k\right\} = \emptyset$}
            {
              Increase trust-region radius: $\Delta_k \gets \Delta_k +1$ until 
$\left\{x \in \Omega_k: \| x - \hat{x} \| \leq \Delta_k\right\} \neq \emptyset$ \label{line:TR}
     }{
	      Set $x^{k+1} \in \ds \argmin_{x \in \Omega_k: \| x - \hat{x} \| \leq \Delta_k} \eta_k(x)$\\
	      Evaluate $f(x^{k+1})$ and set $X^{k+1}\gets X^{k}\cup \{x^{k+1}\} $\\
              \eIf{$f(x^{k+1}) < u_k$}
                  {
		    \underline{\textsf{Upper Bound:}} \\
		    New incumbent $\hat{x} \gets x^{k+1}$ and upper bound $u_{k+1} \gets f(x^{k+1})$\\
         	    Increase trust-region radius $\Delta_{k+1} \gets \Delta_k +1$  \label{line:updateF2}
                  }{
                    No progress: $u_{k+1} \gets u_{k}$ and reduce trust-region radius
                    $\Delta_{k+1} \gets \max \left\{\Delta_{\min},\frac{\Delta_k}{2}\right\}$
                  }
            }
           $k \gets k+1$
}
           \Output{$\hat{x}$, a global minimizer of $f$ on $\Omega$}
\end{algorithm2e}

We note that $G^k$ may not be a subset of $G^{k+1}$, because $\Omega_k$
can contain fewer points as the upper and lower bounds on $f$ are improved. 
However, the following generalization of \thmref{convergence} ensures that
\algref{lookup} still returns a global minimizer of \eqref{eq:MIP-DFO}. 
\begin{theorem}[Convergence of \algref{lookup}]\label{thm:convergence1}
If \asref{1} holds and if $W(G^k)$ includes at least one cut for every $x \in \Omega_k$,
then \algref{lookup} terminates at an
optimal solution $x^*$ of \eqref{eq:MIP-DFO} in finitely many iterations.
\end{theorem}
\begin{proof}
  \algref{lookup} will terminate in a finite number of iterations
  because $\Omega$ is bounded and \linerefa{select_next_x1}
  ensures that $x^k$ is not a previously evaluated element of $\Omega$.
  Because $W(G^k) \subset W(X^k)$, it follows that $\eta_k(x)$ is a valid lower
  bound for $f$ on $\Omega$, and the trust-region mechanism in \linerefa{TR}
  ensures that \algref{lookup} terminates only if $l_{k+1} = u_k$. 
  Therefore, the result is shown.
\end{proof}

\section{Numerical Experiments}\label{sec:numerics}
We now describe numerical experiments performed on multiple versions of
\code{SUCIL}; see \tabref{sucils}.
These methods differ in how $x^{k+1}$ is selected and in the set of points used
within \eqref{model:PLP}.
The last two methods are idealized because they assume access to the true
function value at every point in $\Omega_k$. They are included in order to
provide a best-case performance for a \code{SUCIL} implementation.
In the numerical experiments to follow, we set $\Delta_{\min} \gets 1$ in
\algref{lookup} and use an infinity-norm trust region. All \code{SUCIL}
instances begin by evaluating the starting point $\bar{x}$ and $\left\{
\bar{x} \pm e_1, \ldots, \bar{x} \pm e_n \right\}$ ensuring a finite lower bound at
every point in $\Omega$.

\begin{table}[t]
  \begin{minipage}[b]{0.53\linewidth}
    \small
    \centering
  \begin{tabular}{c|c|c}
    Method & $X$ in \eqref{model:PLP}? & $x^{k+1} =$? \\ \hline \hline
    \code{SUCIL}      & $G^k$ & $\dps  \argmin_{x \in \Omega_k: \| x - \hat{x} \|_{\infty} \leq \Delta_k} \eta_k(x)$\\\hline
    \code{SUCIL-noTR} & $G^k$ & $\dps  \argmin_{x \in \Omega_k} \eta_k(x)$ \\\hline
    \code{SUCIL-ideal1} & $X^k$ & $\dps   \argmin_{x \in \Omega \setminus X^k} f(x)$ \\\hline
    \code{SUCIL-ideal2} & $G^k$ & $\dps   \argmin_{x \in \Omega \setminus X^k} f(x)$ \\\hline
  \end{tabular}
  \caption{Description of how \code{SUCIL} versions choose $X$ in 
  \eqref{model:PLP} and the next iterate
  $x^{k+1}$ (breaking ties in $\argmin$ arbitrarily).
  $G^k$ is defined in \defref{generators}, and $X^{k}$ is all points evaluated before iteration $k$.
  \label{table:sucils}}
  \end{minipage}
  \hfill
  \begin{minipage}[b]{0.44\linewidth}
    \centering
    \includegraphics[width=0.9\linewidth]{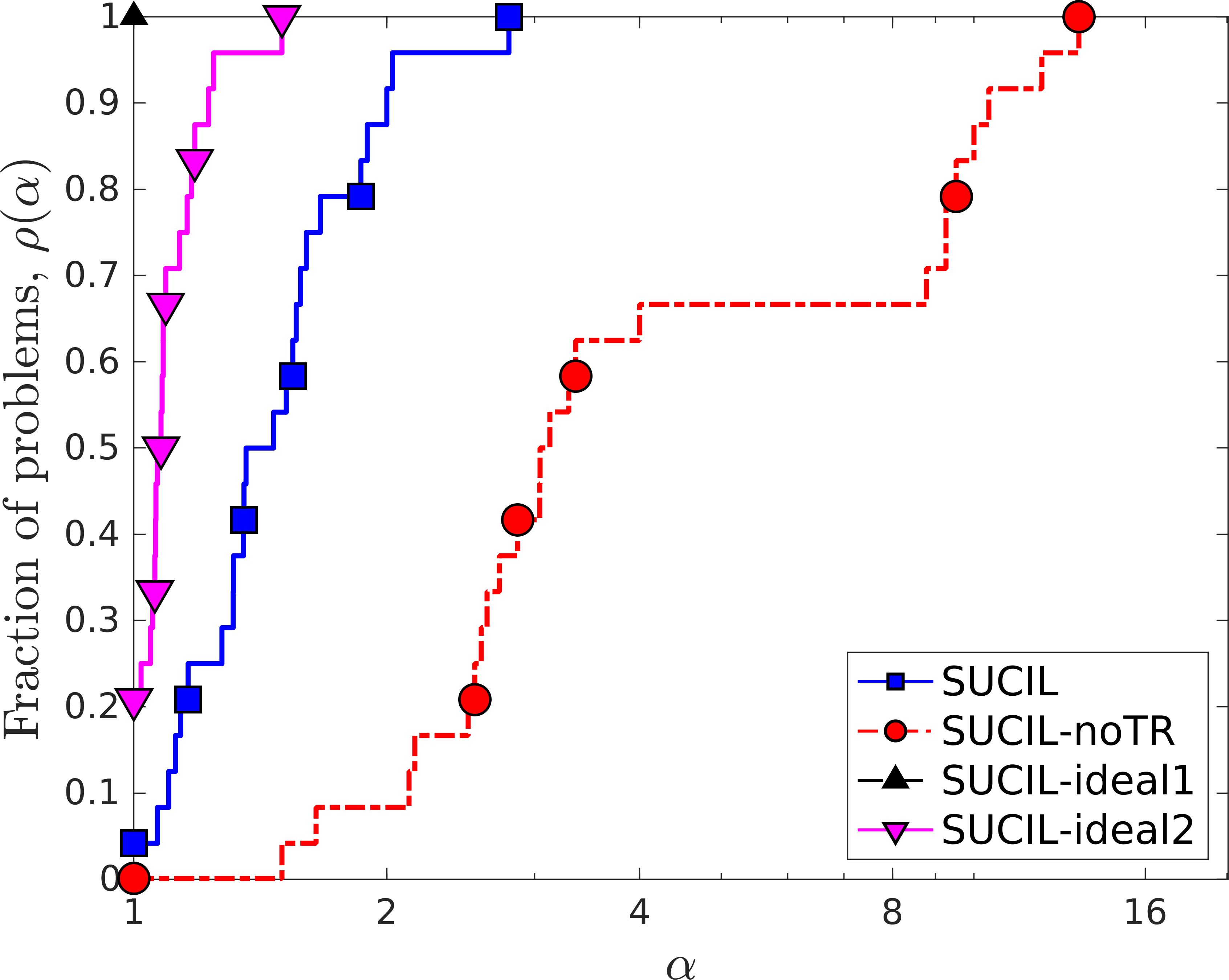}
    \captionof{figure}{Performance profiles for \code{SUCIL}s. Convergence
    measured by number of function evaluations
  before a method terminates with a certificate of global
  optimality.\label{fig:comparisons}}
  \end{minipage}
\end{table}

Below, we compare \code{SUCIL} implementations with a direct-search method,
\code{DFLINT} \cite{Liuzzi2018}; a model-based method, \code{MATSuMoTo}
\cite{Mueller2014a}; and a hybrid method, \code{NOMAD} \cite{Abramson2014}.
We tested the default nonmonotone \code{DFLINT} in MATLAB, as well as the
monotone version, denoted \code{DFLINT-M}.
We tested the default C++ version of \code{NOMAD} (v.3.9.0) as well as the
same version with {\em DISABLE MODELS} set to true, denoted
\code{NOMAD-dm};
the rest of the settings are default.
\code{MATSuMoTo} is a surrogate-model toolbox explicitly designed for
computationally expensive, black-box, global optimization problems. 
Since \code{MATSuMoTo} has a restarting mechanism that ensures that any budget of function
evaluations will be exhausted, 
we input the optimal objective function value to \code{MATSuMoTo} and allowed
it to run (and make as many restarts as required) until the global optimal value
was identified.
The default settings were used: surrogate models using cubic radial basis functions, 
sampling at the minimum of the surrogate, and using an initial 
symmetric Latin hypercube design. 
We performed 20 replications of \code{MATSuMoTo} for each problem instance;
the details of each run are shown in \tabrefs{matsumoto3d}{matsumoto5d} in
\appref{numerics}. We report the floor of the average
number of function evaluations incurred in the last row of these tables and use
this statistic for our comparisons.
A common starting point is given to all methods; the starting point for the
\prob{maxq} and \prob{mxhilb} problems is the global minimizer.
A maximum function evaluation limit of 1,000 is set for all the methods when
$n=4$ or $5$.

We perform numerical experiments minimizing the convex objectives in \tabref{probs} in \appref{probs} on the domains
$\left[ -4,4 \right]^{n}$ for $n \in \{3,4,5\}$ to yield 24 problem instances.
(The last row of \tabref{directions} shows $\left| \Omega \right|$ for these test problems.)
Of note is the \prob{KLT} function that
generalizes the example function from \cite{Kolda2003} that shows how
coordinate search methods can fail to find descent. The function from
\cite{Kolda2003} is itself a modification of the Dennis-Woods function
\cite{Dennis1987}, is strongly convex,
and points $x$ along the line $x_1 = \cdots = x_n$ satisfy
  $f(x) < f(x \pm \epsilon e_i)$
for all $i$ and for all $\epsilon > 0$. 
The problems \prob{CB3II}, \prob{CB3I}, \prob{LQ}, \prob{maxq}, and
\prob{mxhilb} were introduced in
\cite{Haarala2004} and also used in \cite{Liuzzi2018}.
These five problems
are either summation or maximization of generalizations
of simple convex functions, constructed by extending or chaining nonsmooth
convex functions or making smooth functions nonsmooth. The function \prob{LQ}
takes a global minimum at any $x \in [0,1]^n \cap \Z^n $ that does not
have zeros in consecutive coordinates. For example, for $n=3$, the points 
$[0,1,0]^\T, [0,1,1]^\T, [1,0,1]^\T, [1,1,0]^\T$, and $[1,1,1]^\T$ are optimal
but $[0,0,0]^\T, [0,0,1]^\T$, and $[1,0,0]^\T$ are not.

If there are relatively few points in $\Omega$ and the time required to evaluate $f$
is small (as for our test instances), one could argue that an enumerative procedure
itself could solve the problem in a reasonable time. However, we use
these 
instances to thoroughly examine the behavior that might be seen on 
expensive-to-evaluate black-box functions. 
Therefore, we compare methods using performance profiles \cite{DolaMore:02} that
are based on the number of function evaluations required to satisfy the
respective convergence criterion.
For each method $s$, $\rho_s(\alpha) = \frac{\lvert \{p \in P: r_{p,s}\leq \alpha \} \rvert}{\lvert P
\rvert}$, for a scalar $\alpha \geq 1$, $P$ is the collection of 
problems, and $r_{p,s}=\frac{N_{p,s}}{\min_{s \in S}\{N_{p,s}\}}$ is the
performance ratio. We consider two measures of $N_{p,s}$:
(1) the number of function evaluations before a method $s$ terminates on a
problem $p$ and (2) the number of function
evaluations taken by method $s$ to evaluate a global minimizer on problem $p$.

\figref{comparisons} compares the number of evaluations required for four
implementations of \code{SUCIL} to terminate (with a certificate of optimality)
on the set of test problems. While \code{SUCIL-ideal1} is no slower than
any other implementation on all the test problems, it is not a realistic method
in that it evaluates points based on their known function values. \code{SUCIL}
requires no more than three times the evaluations as \code{SUCIL-ideal1} for
the set of test problems.  We do observe
that using a trust region in \code{SUCIL} is a significant advantage. For many
of the problems considered, \code{SUCIL-noTR} spent many function evaluations
in the corners of $\Omega$.

As a point of comparison
with the results in \figref{comparisons}, a different
estimate of the number of function evaluations (or primitive directions
explorations) required for the proof of optimality for our instances can be
seen in \tabref{directions}, in columns corresponding to $n \in \{3,4,5\}$ and
$k=4$. As evident from the results in \tabrefs{3dComparisons}{5dComparisons}, 
our method incurs a remarkably low number of
function evaluations, which can be attributed to exploitation of convexity and
subsequent formation of the underestimators, as explained in \secref{background}.

\begin{figure}[t]
  \begin{center}
    \includegraphics[width=0.45\linewidth]{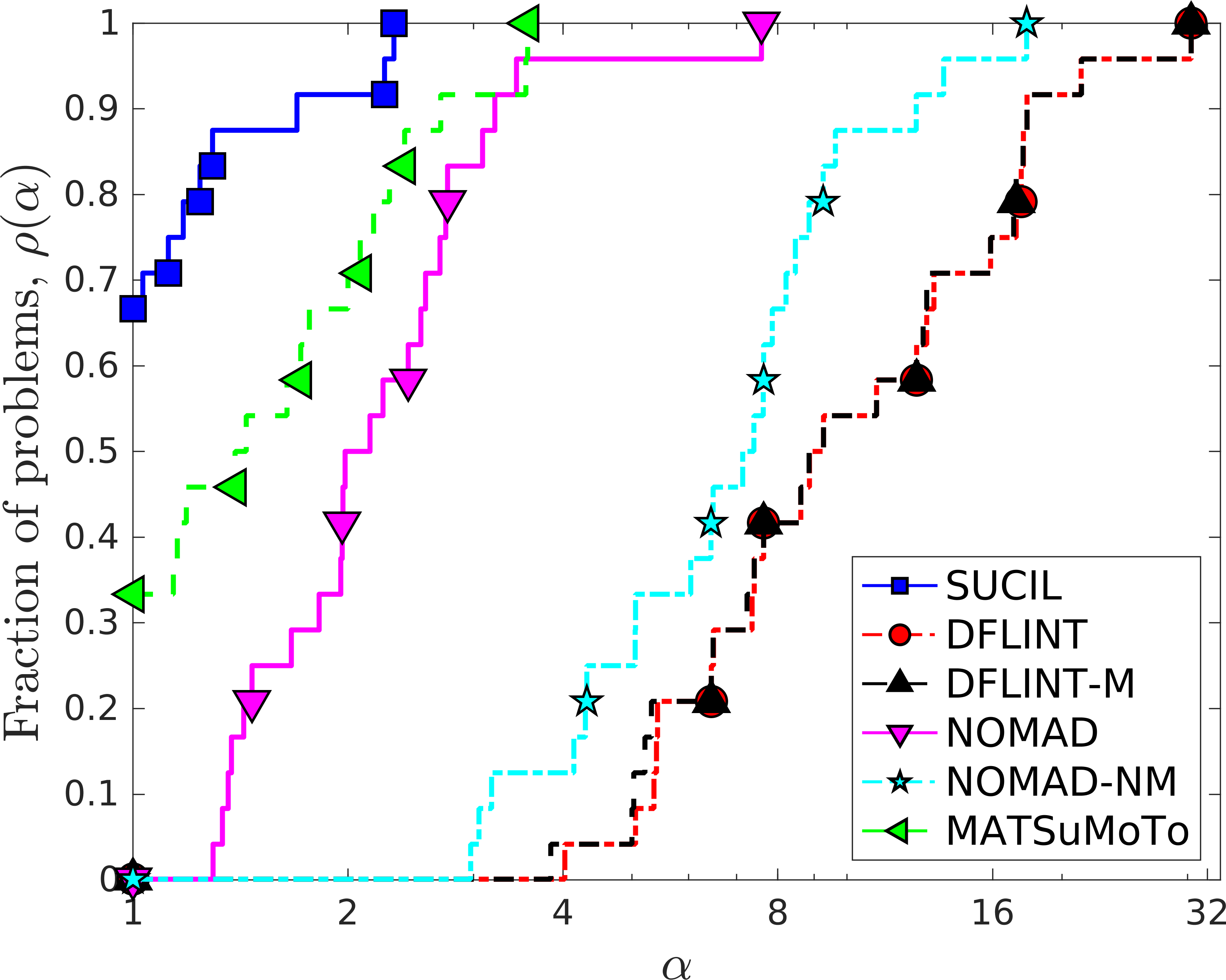}
    \hfill
    \includegraphics[width=0.45\linewidth]{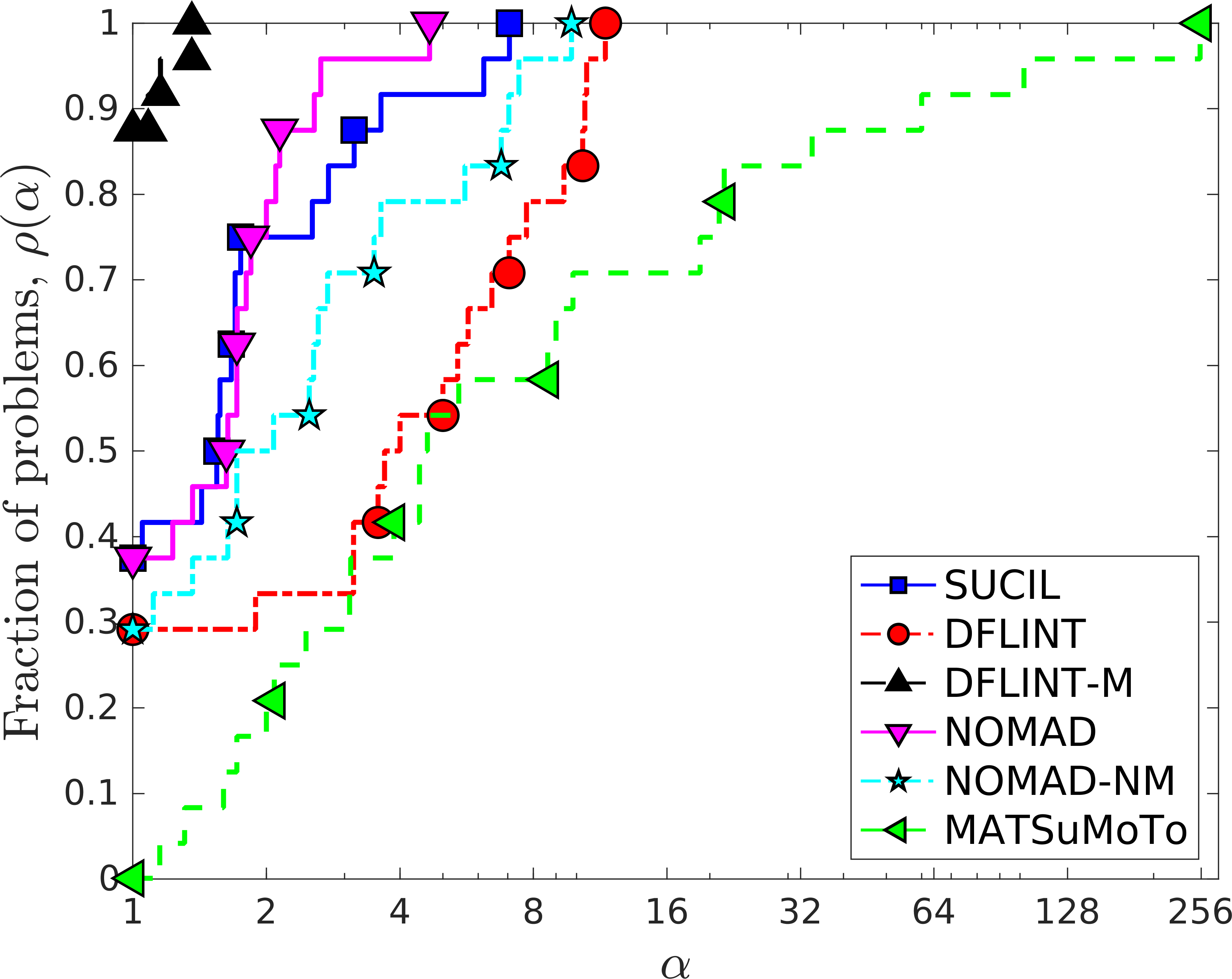}
  \end{center}
  \caption{Performance profiles of different methods solving 
  24 problem instances. Left compares the number of evaluations until a
  method terminates; right compares the number of evaluations before a
  method first evaluates a global minimizer.
  Performance for each solver on each problem instance can be found in
  \tabrefs{3dComparisons}{5dComparisons} in \appref{numerics}.
  \label{fig:perfProf}}
\end{figure}

We now analyze the performance of \code{SUCIL} compared with the other methods.
\figref{perfProf} (left) shows the performance profiles for methods to terminate   
on the 24 test problems; \figref{perfProf} (right) compares the number of function
evaluations required before each method first evaluates a global minimizer.
This comparison is nontrivial because each solver has its own design
considerations and notions of local optimality. Also, the other solvers
do not assume convexity of the problem or exploit it. Hence, our results merely
demonstrate that \code{SUCIL} provably converges to a global optimum and uses
fewer function evaluations because of its exploitation of convexity.
\figref{perfProf} shows that our algorithm requires the least number of function evaluations
for more than 65\% of the instances and provides a global optimality
certificate, in addition. In reaching the global optimal solution quickly, however,
DFLINT-M wins for more than 85\% of the instances. Although \code{SUCIL} is
not particularly designed to greedily descend to the global optimum, it
is still competitive with the rest of the methods on this front.

\section{Discussion}\label{sec:chall}

\begin{figure}
  \begin{center}
    \includegraphics[width=0.7\linewidth]{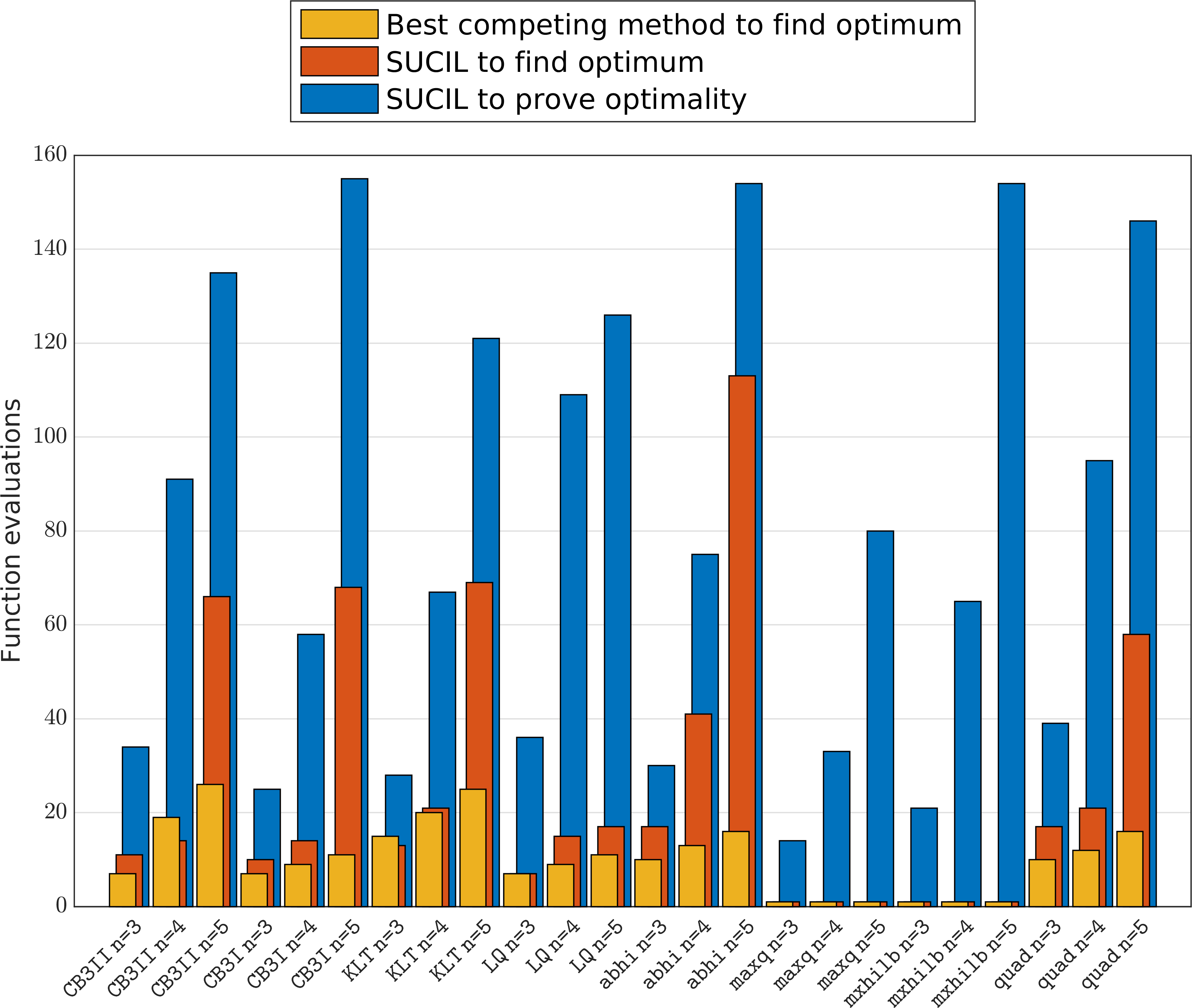}
  \end{center}
  \caption{Number of function evaluations before \code{SUCIL} first
  identifies a global minimum and evaluations required to prove its global
  optimality. The fewest number of evaluations required by any of
  \code{DFLINT}, \code{DFLINT-M}, \code{NOMAD}, \code{NOMAD-dm},
  and \code{MATSuMoTo} is shown for comparison.
  \label{fig:OptVsProof}}
\end{figure}

The order of results in this paper tells the story of how we arrived at the
implementation of \code{SUCIL}. We first attempted to classify where linear
interpolation models provide lower bounds for convex functions, yielding the results
in \secref{cuts}; we then proved that such linear functions can underlie a convergent
algorithm, as in \secref{algm}. We initially modeled the secants and the
conditions in which they are valid as an MILP, as
in \secref{tols}. After observing that the number of variables in the MILP model
was larger than the number of points in the domain  $\Omega$, we were motivated to develop
the enumerative model in \secref{enum}.

Our computational developments expose a number of fundamental challenges for
integer derivative-free optimization.
The complexity of our piecewise linear model \eqref{model:PLP} is made worse by
the fact that each secant function is valid only in the union of $n+1$ cones
$\cU^{\bi}$, resulting in conditional cuts. 
We note that it may not be possible to derive
{\em unconditional cuts}, that is, cuts that are valid in the whole domain $\Omega$.
For example, we might initially consider
secants interpolating a convex $f$ at the $n+1$ points $x \in \Z^n$ and $x \pm e_i \in \Z^n$,
where for every $i$ we can choose either $+$ or $-$. 
Such points form a unit simplex that has no integer 
points in its interior. Consequently, one might suspect that the resulting cut
is valid everywhere in $\Omega$. However,
the following example shows that the resulting cut is {\em not}
unconditionally valid. Consider $f(x) = x_1^2 - x_1 x_2 + x_2^2$ and the set of points
$\{ [1,1]^\T, [0,1]^\T, [1,0]^\T \}$. It follows that $f(x)=1$ at these points, and hence the
unique interpolating secant function is the constant function, $m(x) = 1$. Now consider
the point $x=[0,0]^\T$ for which $f(x)=0$, which is not underestimated by $m(x) = 1$.
Another limitation of our method is that it will have to evaluate all feasible points
when $\Omega = \{0,1\}^n$ (a pure binary domain) for convergence because no point in
$\Omega \setminus X^{\bi}$ belongs to $\cU^{\bi}$, where $X^{\bi}$ is an
arbitrary poised set of $n+1$ points in $\{0,1\}^n$.

In \figref{OptVsProof} we show the number of function evaluations needed to
first evaluate a global minimizer and the additional number of evaluations used
to prove it is a global minimizer. As is common, the effort required to
certify optimality can be significantly larger than the cost of finding the
optimum.  In terms of number of function evaluations required, the proof of
optimality is even more time consuming. Because the iterations where $X^k$ or
$G^k$ is large require checking many potential secant functions, in
\code{SUCIL} the computational cost of iterations can differ by orders of
magnitude as the algorithm progresses.

Although our method provides a practical iterative way to check sufficiency of a
set of points (optimality conditions) for a given convex instance, each
iteration involves construction and evaluation of a large number
of combinations of different $n+1$ points, which limits the scalability of
\algref{lookup} in solving instances of higher dimensions.
Yet, in our numerical experiments, we observe that only a small fraction of the
total cuts evaluated are useful.  We call $(c^{\bi},b^{\bi})$ an {\em
updating} cut at an iteration $k$ if there exists an $x \in \Omega_k$ such that
$m^{\bi}(x) > \eta_k(x)$, that is, a cut that improves the lower bound at at least
one $x \in \Omega_k$.  In addition, if $m^{\bi}(x) \geq u_k$, we call it a {\em
pruning} cut. A pruning cut helps eliminate points to be considered in the next
iteration ($\Omega_{k+1}$). \figref{updatePrune} shows the number of
updating and pruning cuts generated per iteration of \code{SUCIL} when minimizing
\prob{quad} on $[-4,4]^3 \cap \Z^3$. The fact that few cuts prune a point or update the
lower bound at any point where the minimum could be suggests that there may be
some way to exclude a large set of multi-indices from consideration, possibly
yielding dramatic computational savings.

\begin{figure}
  \begin{center}
    \includegraphics[width=0.7\linewidth]{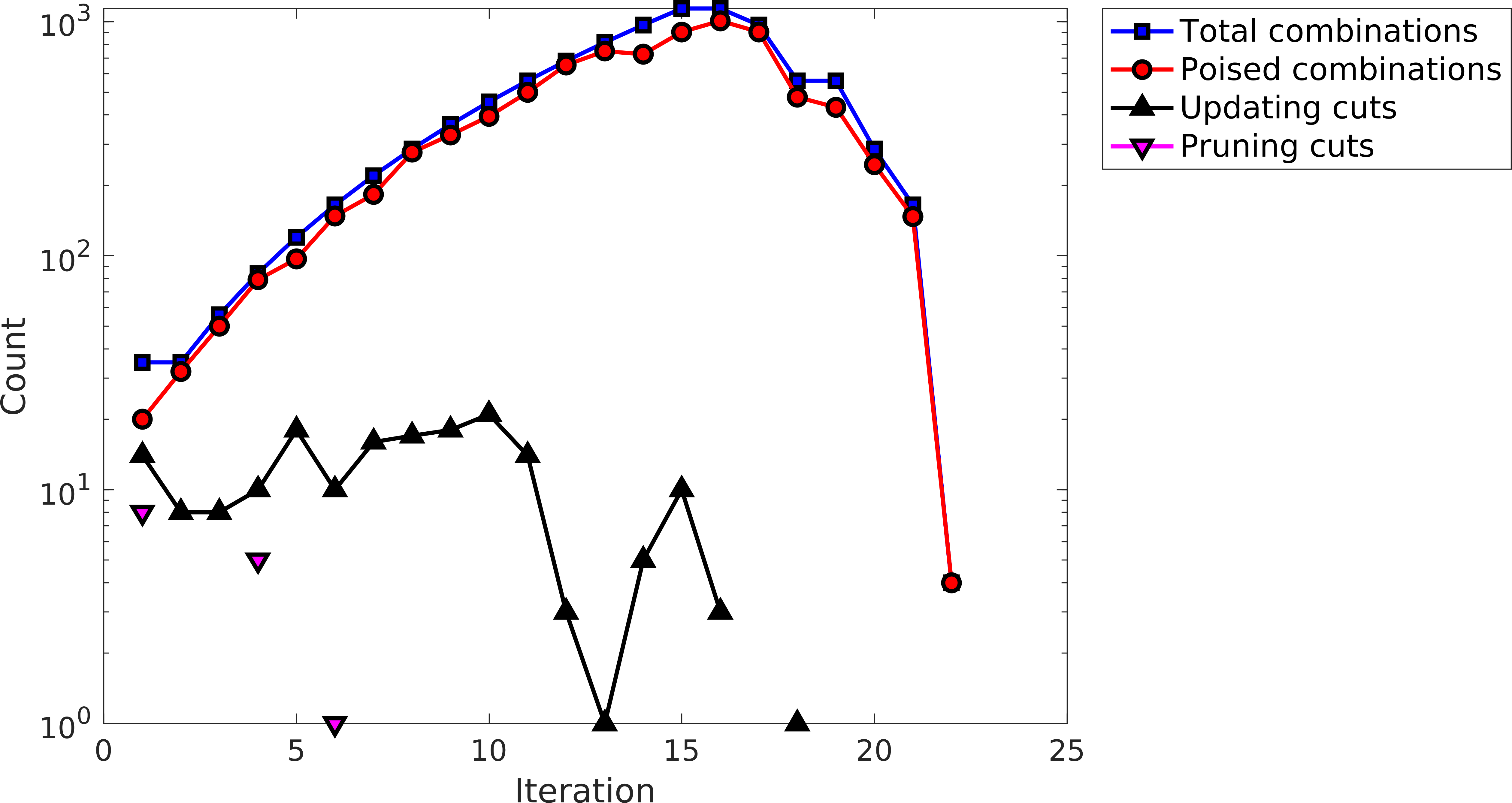}
    \caption{Number of total and poised combinations of
    $n+1$ points, the secant functions that update, and the secant functions that
    prune at least one point when minimizing \prob{quad} on $\Omega = [-4,4]^3 \cap \Z^3$ using
    \code{SUCIL}. (Markers are removed when there is no updating or no
    pruning cut in an iteration.)\label{fig:updatePrune}}
  \end{center}
\end{figure}

Ideally, we would like to evaluate only the combinations that yield
updating or pruning cuts. However, this approach requires the solution of a
separate problem that we believe is especially hard to solve. Even the following
simpler problem of finding a pruning cut at a given candidate point seems
difficult.

\begin{problem} \label{prob:cutSep}
Given a point $\bar{x} \in \Z^n$, a set of (integer) points $X$ where $f$ has
been evaluated, and scalar $u$, find a cut that prunes $\bar{x}$. That is, find a
multi-index $\bi$ such that $\bar{x} \in \cU^{\bi}$ and ${(c^{\bi})}^\T
\bar{x} + b^{\bi} \geq u$, and $(c^{\bi},b^{\bi})$ solves \eqref{eq:interp},
or show that no such multi-index $\bi$ exists.
\end{problem}

If we choose a small
subset $\bar{X^k}$ of $X^k$ to form $W(\bar{X}^k)$, the \code{SUCIL} algorithm
can end up using a
large number of function evaluations to obtain a certificate of optimality. The
reason is that points are evaluated that would be ruled worse than optimal if
secants were built by using all combinations of points in $X^k$. 
This situation occurred when setting $\bar{X^k}$ to be a random subset of $X^k$, a subset
of the points closest to $\hat{x}$, or a subset of points with best function
values. 
Using $G^k$ avoids discarding too many
points from $X^k$; but we observe a significant increase in  
$\left| W(G^k) \right|$, and thus we incur heavy computational costs during some iterations. The
wall-clock time required per iteration for solving instances of dimension less
than 5 in our setup is not significant, but we present the same for
5-dimensional instances using \code{SUCIL} on a 96-core Intel Xeon computer with 1.5~TB of RAM. 
The complexity of our approach is better quantified by counting
the number of combinations of points (or potential secants) 
considered at iteration $k$. Using $G^k$, we typically produce a
strict subset
of all possible combinations in such a way that the size of $W(G^k)$
decreases during the later iterations. This is shown in
\figref{walltime_and_secants}: the number of
secants added per iteration for all 5-dimensional test instances using
\code{SUCIL}. Once $\Omega_k$, the number of points with $\eta(x)$ less than
$f(\hat{x})$, starts decreasing, so do $G^k$ and $\left| W(G^k) \right|$. 
In general, it is difficult to predict when the
number of combinations (or the wall-clock time curve) would be at the peak, but we
suspect this peak will be worse as $n$ increases, by both the size and the iteration
number where it occurs. This limits the applicability of the current
implementation of \code{SUCIL} on higher-dimensional problems.
\begin{figure}
  \begin{center}
    \includegraphics[width=0.45\linewidth]{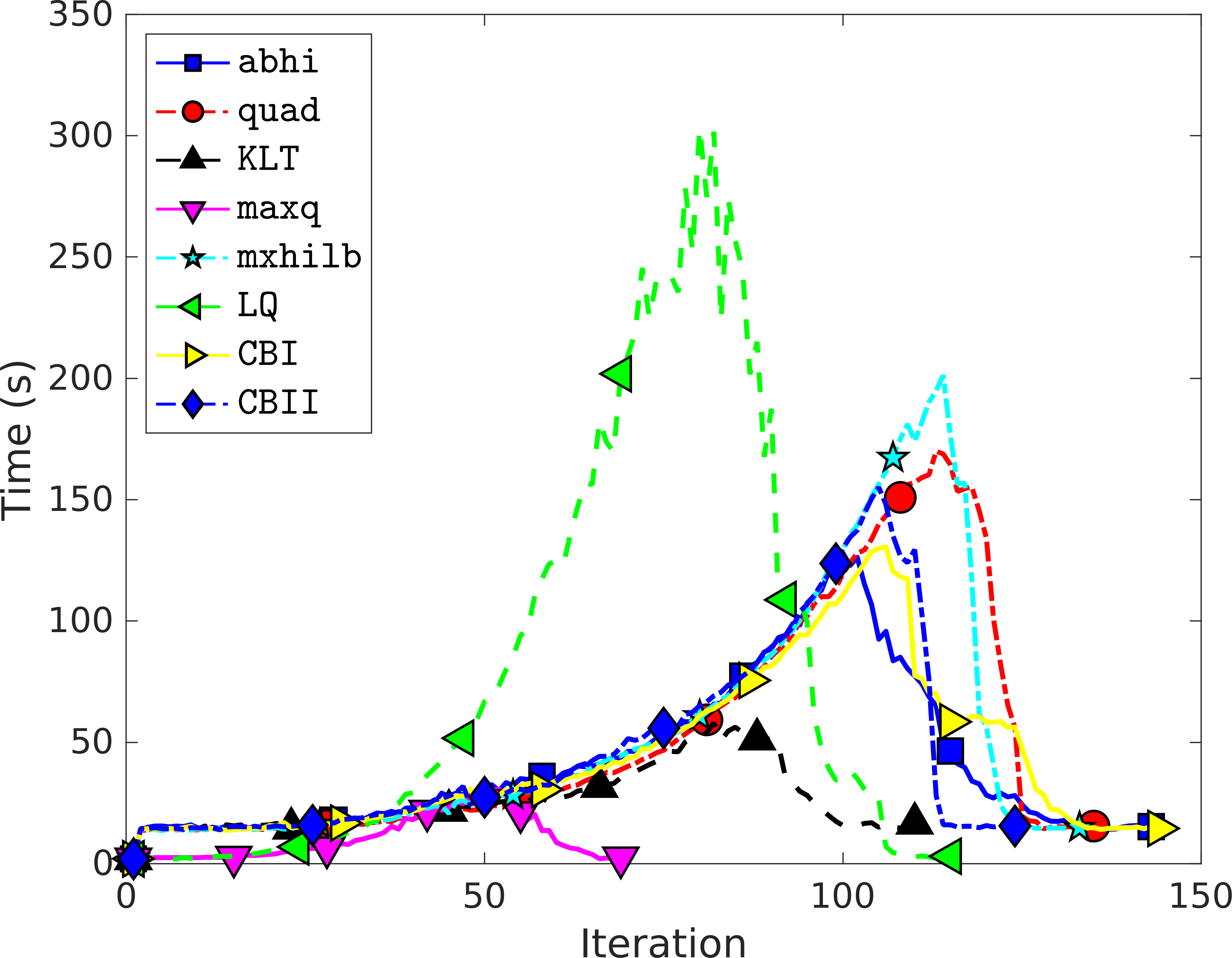}
    \hfill
    \includegraphics[width=0.45\linewidth]{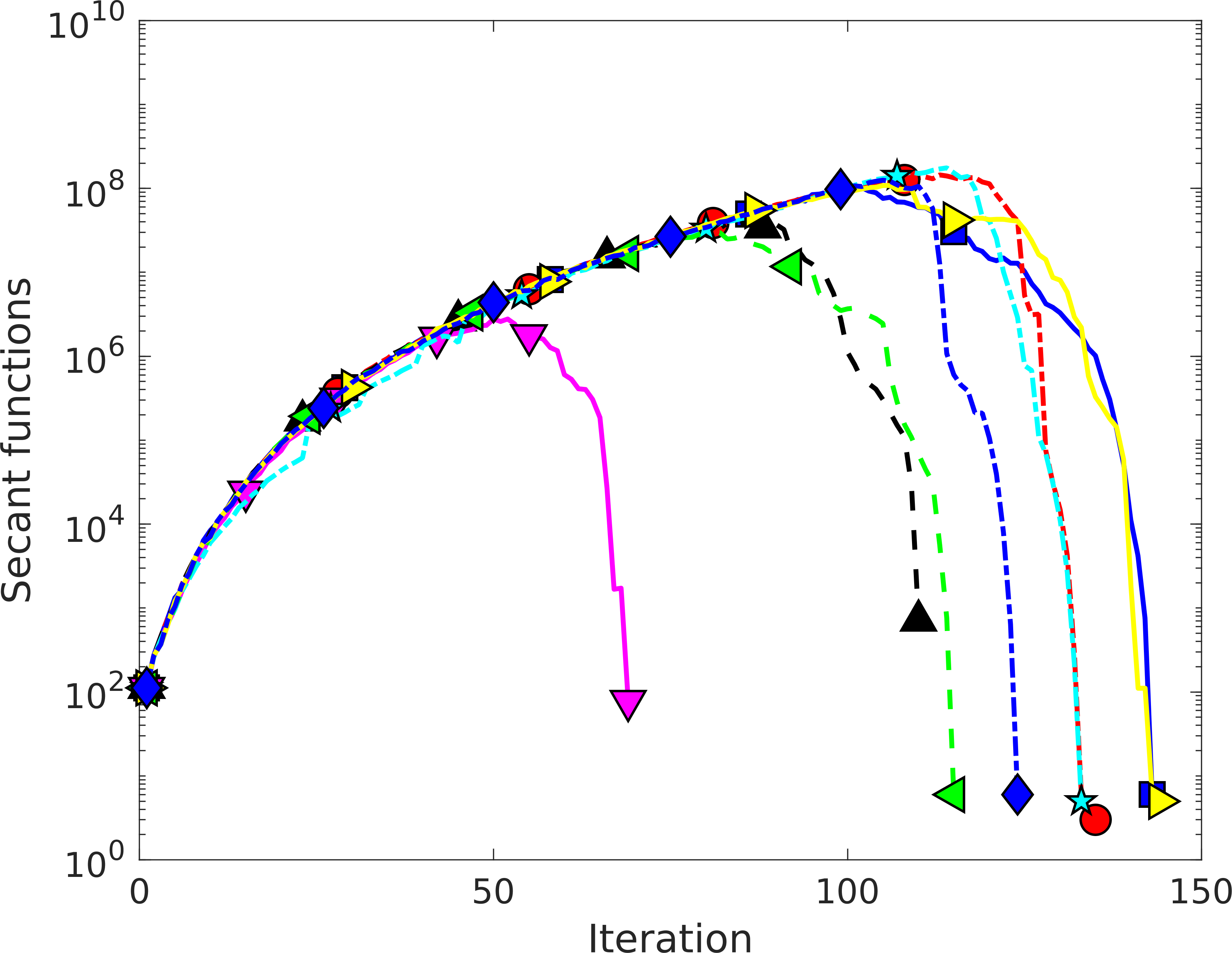}
  \end{center}
  \caption{Wall-clock time recorded and number of secants constructed per
  iteration of \code{SUCIL} for 8 convex test problems on
  $\Omega = [-4,4]^5 \cap \Z^5$. \label{fig:walltime_and_secants}}
\end{figure}

Again, since nearly all cuts in $W(G^k)$ do not update $\eta(x)$ at any point in
$\Omega_k$ (see \figref{updatePrune}), we believe there may be some approach for
intelligently selecting points from $X^k$ using their geometry, their function
values, and distance from $\hat{x}$ that will rule some multi-indices $\bi$ as
unnecessary to consider. We did attempt to identify minimal sets of points that
were necessary for \code{SUCIL} to certify optimality for a variety of
$n=2$ test cases, but no general rule was apparent.
 
We note that the storage requirements for the enumerative model may be prohibitive,
even for moderate problem sizes. For example, an array storing the value of
$\eta(x)$ as an 8-byte scalar for all $x \in \Omega = [-10, 10]^{10} \cap
\Z^{10}$ would require over 200~GB of storage. 

Ultimately, we believe further insights are yet to be discovered that will
facilitate better algorithms for minimizing convex functions on integer domains.  

\section*{Acknowledgements}
We are grateful to Eric Ni for his insights on derivative-free algorithms for 
unrelaxable integer variables. 
Sven Leyffer also wishes to acknowledge the insightful
discussions on an early draft of this work during the Dagstuhl seminar 18081.
This material is based upon work supported by the 
applied mathematics and SciDAC activities of the
Office of Advanced Scientific Computing Research,
Office of Science, 
U.S.\ Department of Energy,
under Contract DE-AC02-06CH11357.

\appendix
\renewcommand{\thetable}{\Alph{section}\arabic{table}}
\setcounter{figure}{0}
\renewcommand{\thefigure}{\Alph{section}\arabic{figure}}
\section{Test Problems}\label{app:probs} \setcounter{table}{0}
\renewcommand{\arraystretch}{1.5}
\begin{table}[h]
  \scriptsize
  \captionsetup{font=footnotesize}
\center
\caption{Set of convex test problems.\label{table:probs}}
\begin{tabular}{|c|c|c|c|c|}
\hline
Name & Expression & $f(x^*)$ & $x^*$ \\ \hline \hline
\prob {CB3II} \cite{Charalambous1978} & $\dps\max \left \{ \sum_{i=1}^{n-1} x_i^4 + x_{i+1}^2, \sum_{i=1}^{n-1} \left (2-x_i \right )^2 + \left (2-x_{i+1} \right )^2, \sum_{i=1}^{n-1} 2e^{-x_i + x_{i+1}} \right \}$ & $2(n-1)$ & $e$ \\ \hline
\prob {CB3I} \cite{Charalambous1976non} & $\dps\sum_{i=1}^{n-1}\max \left \{ x_i^4 + x_{i+1}^2, \left (2-x_i \right )^2 + \left (2-x_{i+1} \right)^2, 2e^{-x_i + x_{i+1}} \right \}$ & $2(n-1)$ & $e$ \\ \hline
\prob {KLT} \cite{Kolda2003} & $\dps\max_{i \in \left\{ 1,\ldots,n \right\}} \left\{ \left\| x - c_i -2e \right\|^2 \right\}, c_i = 2 e_i - e$ & $n$ & $2e$ \\ \hline
\prob {LQ} \cite{Haarala2004} & $\dps\sum_{i=1}^{n-1}\max \left \{ -x_i-x_{i+1}, -x_i-x_{i+1} + x^2_i + x^2_{i+1} -1 \right \}$ & $-(n-1)$ & many  \\ \hline
\prob {abhi} \cite{Abhishek2010} & \vtop{\hbox{\strut $\dps\sum_{i=1}^n \left[ 64 \left(c_1(x_i-2)-c_2(x_{i+1} -2) \right)^2 + \left(c_2(x_i-2) - c_1(x_{i+1} -2) \right)^2 \right]$,} \hbox{\strut \hspace{5.5em} $c_1=\cos\left( \frac{\pi}{8} \right), c_2=\sin\left( \frac{\pi}{8} \right)$} \vspace{0.5em}}  & 0 & $2e$ \\ \hline
\prob {maxq} \cite{Haarala2004} & $\dps\max_{i \in \left \{ 1,\ldots,n \right \}} \left\{ x^2_i \right\}$  & 0 & 0 \\ \hline
\prob {mxhilb} \cite{Kiwiel1989}  & $\dps\max_{i \in \left\{ 1,\ldots,n \right\}} \left\{ \sum_{j=1}^n \left \lvert \frac{x_j}{i+j-1} \right \rvert \right\}$ & 0 & 0 \\ \hline
\prob {quad} & $\dps\sum_{i=1}^n \left (x_i-2 \right )^2 $ & 0 & $2e$ \\ \hline
\end{tabular}
\end{table}
\renewcommand{\arraystretch}{1}

\FloatBarrier 
\section{Performance of MILP Model}\label{app:mip} \setcounter{table}{0}

\begin{table}
  \scriptsize
  \captionsetup{font=footnotesize}
  \caption{Characteristics of the first 12 instances of \eqref{model:CPFull}
  generated by \algref{framework} minimizing the convex quadratic function
  \prob{abhi} on $\Omega=[-2,2]^3 \cap \Z^3$.
  \eqref{model:CPFull} instances are generated by AMPL and solved by
  CPLEX; times are the mean of five replications.\label{table:mipExpo}}
  \center
\begin{tabular}{|r|r|r|r|r|r|r|r|r|r|r|r|r|r|}
\hline
\textit{k} & \textit{sHyp} & \textit{LB}
& \textit{UB} & \textit{time} & \textit{simIter} & \textit{nodes}
& \textit{bVars} & \textit{cVars} & \textit{cons} & \textit{$\hat{x}$} \\ \hline
1  & 20   & -616.3 & 79.9 & 0.1 & 374    & 0     & 335   & 268   & 960              & $[2; 2;-2]$ \\ \hline
2  & 52   & -555.1 & 79.9 & 3.9 & 13,180  & 8,089  & 847   & 685   & 2,466          & $[2; 2;-1]$ \\ \hline
3  & 100  & -475.2 & 44.7 & 10.9 & 31,423  & 8,555  & 1,615  & 1,310  & 4,724       & $[2; 1;-2]$ \\ \hline
4  & 172  & -434.4 & 44.7 & 7.3 & 19,267  & 1,728  & 2,767  & 2,247  & 8,110        & $[1; 2;-2]$ \\ \hline
5  & 276  & -413.9 & 19.1 & 30.8 & 68,264  & 5,874  & 4,431  & 3,600  & 13,000      & $[2; 1;-1]$ \\ \hline
6  & 418  & -373.1 & 19.1 & 95.1 & 84,031  & 7,933  & 6,703  & 5,447  & 19,676      & $[1; 2;-1]$ \\ \hline
7  & 611  & -311.7 & 19.1 & 59.6 & 83,102  & 5,440  & 9,791  & 7,957  & 28,749      & $[2;-2;-2]$ \\ \hline
8  & 866  & -293.2 & 19.1 & 99.6 & 86,318  & 3,933  & 13,871 & 11,273 & 40,736      & $[1; 1;-2]$ \\ \hline
9  & 1,196 & -232.0 & 19.1 & 154.2 & 84,440  & 4,568  & 19,151 & 15,564 & 56,248    & $[1; 1;-1]$ \\ \hline
10 & 1,532 & -199.5 & 19.1 & 452.3 & 235,473 & 6,400  & 24,527 & 19,933 & 72,042    & $[2;-2;-1]$ \\ \hline
11 & 2,038 & -192.9 & 19.1 & 1,006 & 387,491 & 9,686  & 32,623 & 26,512 & 95,826    & $[2;-1;-2]$ \\ \hline
12 & 2,605 & -140.9 & 19.1 & 964.3 & 455,939 & 29,279 & 41,695 & 33,884 & 122,477   & $[1;-1;-2]$ \\ \hline
\end{tabular}
\end{table}

\tabref{mipExpo} shows the size of the MILP model at 
each iteration and the computational effort required for solving it.
The column {\em k} refers to the iteration of \algref{framework},
{\em sHyp} denotes the number of secants (i.e., $\left| W(X) \right|$),
and {\em LB} and {\em UB} give the lower and upper bound on $f$ on $\Omega$, respectively. We show the
computational effort needed to solve each MILP via {\em time}, the
mean solution time (in seconds) for $5$ replications; {\em simIter}, the number of simplex iterations; and {\em
nodes}, the number of branch-and-bound nodes explored by the MILP solver.
The size of each MILP (after presolve) is
shown in terms of {\em bVars}, the number of binary variables; {\em
cVars}, the number of continuous variables; and {\em cons}, the number of
constraints. \tabref{mipExpo} also shows the optimal solution $\hat{x}$ of each
MILP. These experiments were performed by using CPLEX (v.12.6.1.0) on a 2.20~GHz, 12-core Intel
Xeon computer with 64~GB of RAM. 
For this small problem we see that the size of the MILP grows
exponentially as the iterations proceed, which results in an exponential growth
in solution time as illustrated in \figref{MIPExpo}. The iteration 13 MILP was
not solved after $30$ minutes.

\FloatBarrier 
\section{Detailed Numerical Results}\label{app:numerics} \setcounter{table}{0}
\tabrefs{3dComparisons}{matsumoto5d} contain detailed numerical results for the
interested reader. Note that some
solvers do not respect the given budget of function evaluations.
We have used a different stopping criterion for \code{MATSuMoTo}: it is set to stop only
when a point with the optimal value has been identified. Also, although the
global minimum is the starting point for \prob{maxq} and \prob{mxhilb},
\code{MATSuMoTo} instead uses its initial symmetric Latin hypercube design.
The last row of \tabref{directions} in \secref{background} shows $\left| \Omega \right|$ for these problems.

\begin{table}[htbp]
  \scriptsize
  \captionsetup{font=footnotesize}
  \caption{Number of function evaluations before solvers terminate for $n=3$ test problems. Parentheses show number of evaluations until $x^*$ is evaluated.\label{table:3dComparisons}}
\center
\begin{tabular}{|c|r|r|r|r|r|r|r|r|}
\hline
& \code{SUCIL-ideal1} & \code{SUCIL} & \code{DFLINT} & \code{DFLINT-M} & \code{NOMAD} & \code{NOMAD-dm} & \code{MATSuMoTo}\\ \hline
\prob{abhi}      & 19 (8) & 30 (17) & 161 (57)  & 150 (10) & 59 (20) & 129 (56) & 60 (31)\\ \hline 
\prob{quad}      & 25 (8) & 39 (17) & 157 (54)  & 150 (10) & 53 (18) & 119 (35) & 45 (16)\\ \hline 
\prob{KLT} & 22 (8) & 28 (13) & 152 (48)  & 146 (15) & 51 (24) & 116 (34) & 46 (17)\\ \hline 
\prob{maxq}      & 14 (1) & 14 (1)  & 181 (1)   & 181 (1)  & 34 (1)  & 107 (1)  & 50 (21)\\ \hline 
\prob{mxhilb}    & 16 (1) & 21 (1)  & 181 (1)   & 181 (1)  & 35 (1)  & 106 (1)  & 48 (19)\\ \hline 
\prob{LQ}   & 18 (8) & 36 (7)  & 182 (7)   & 181 (7)  & 48 (7)  & 107 (7)  & 41 (12)\\ \hline 
\prob{CB3I} & 22 (8) & 25 (10) & 184 (22)  & 181 (7)  & 56 (12) & 108 (12) & 60 (31)\\ \hline 
\prob{CB3II}& 22 (8) & 34 (11) & 184 (22)  & 181 (7)  & 44 (12) & 108 (12) & 60 (31)\\ \hline 
\end{tabular}
\end{table}

\begin{table}[htbp]
  \scriptsize
  \captionsetup{font=footnotesize}
  \caption{Number of function evaluations before solvers terminate for $n=4$ test problems. Parentheses show number of evaluations until $x^*$ is evaluated.\label{table:4dComparisons}}
\center
\begin{tabular}{|c|r|r|r|r|r|r|r|r|r|}
\hline
& \code{SUCIL-ideal1} & \code{SUCIL} & \code{DFLINT} & \code{DFLINT-M} & \code{NOMAD} & \code{NOMAD-dm} & \code{MATSuMoTo}\\ \hline
\prob{abhi} 	   & 45 (10) & 75 (41)  & 993 (137) & 959 (13)  & 110 (16) & 453 (88) & 89  (60)\\ \hline 
\prob{quad} 	   & 51 (10) & 95 (21)  & 982 (124) & 959 (13)  & 111 (12) & 460 (89) & 56  (25)\\ \hline 
\prob{KLT} 	   & 51 (10) & 67 (21)  & 954 (141) & 953 (20)  & 116 (42) & 457 (55) & 54  (23)\\ \hline 
\prob{maxq} 	   & 30 (1)  & 33 (1)   & 1,001 (1)  & 1,001 (1)  & 91  (1)  & 451 (1)  & 89  (60)\\ \hline 
\prob{mxhilb}        & 32 (1)  & 65 (1)   & 1,001 (1)  & 1,001 (1)  & 90  (1)  & 450 (1)  & 63  (34)\\ \hline 
\prob{LQ}       & 39 (10) & 109 (15) & 1,000 (17) & 1,001 (9)  & 145 (9)  & 453 (10) & 47  (18)\\ \hline 
\prob{CB3I}     & 50 (10) & 58 (14)  & 1,001 (58) & 1,001 (9)  & 149 (42) & 456 (23) & 126 (81)\\ \hline 
\prob{CB3II}    & 48 (10) & 91 (14)  & 1,001 (50) & 1,001 (19) & 125 (30) & 460 (35) & 131 (76)\\ \hline 
\end{tabular}
\end{table}

\begin{table}[htbp]
  \scriptsize
  \captionsetup{font=footnotesize}
  \caption{Number of function evaluations before solvers terminate for $n=5$ test problems. Parentheses show number of evaluations until $x^*$ is evaluated.\label{table:5dComparisons}}
\center
\begin{tabular}{|c|r|r|r|r|r|r|r|r|r|}
\hline
& \code{SUCIL-ideal1} & \code{SUCIL} & \code{DFLINT} & \code{DFLINT-M} & \code{NOMAD} & \code{NOMAD-dm} & \code{MATSuMoTo}\\ \hline
\prob{abhi} 	    & 105 (12) & 154 (113) & 1,001 (167) & 1,000 (16) & 301 (41) & 1,000 (156) & 214 (157)\\ \hline 
\prob{quad} 	    & 108 (12) & 146 (58)  & 1,001 (186) & 1,000 (16) & 286 (26) & 1,000 (58) & 113 (62) \\ \hline 
\prob{KLT}      & 108 (12) & 121 (69)  & 1,001 (193) & 1,001 (25) & 296 (43) & 1,000 (176) & 108 (77) \\ \hline 
\prob{maxq} 	    & 75 (1)   & 80 (1)    & 1,001 (1)   & 1,001 (1)  & 257 (1)  & 1,000 (1)   & 284 (255)\\ \hline 
\prob{mxhilb}    	    & 96 (1)   & 154 (1)   & 1,001 (1)   & 1,001 (1)  & 257 (1)  & 1,000 (1)   & 131 (102)\\ \hline 
\prob{LQ}        & 83 (12)  & 126 (17)  & 1,001 (44)  & 1,001 (11) & 425 (15) & 1,000 (15)  & 56 (27)  \\ \hline 
\prob{CB3I}      & 114 (12) & 155 (68)  & 1,001 (103) & 1,001 (11) & 417 (18) & 1,000 (18)  & 266 (237)\\ \hline 
\prob{CB3II}     & 100 (12) & 135 (66)  & 1,001 (130) & 1,001 (26) & 465 (69) & 1,000 (54)  & 281 (224)\\ \hline 
\end{tabular}
\end{table}

\begin{table}[htbp]
  \scriptsize
  \captionsetup{font=footnotesize}
  \caption{Number of function evaluations taken by 20 replications of \code{MATSuMoTo} for each of the $8$ convex test problems for $n=3$.\label{table:matsumoto3d}}
\center
\begin{tabular}{|c|r|r|r|r|r|r|r|r|}
\hline
& \prob{abhi} & \prob{quad} & \prob{KLT} & \prob{maxq} & \prob{mxhilb} & \prob{LQ} & \prob{CB3I} & \prob{CB3II} \\ \hline
1  & 48 (19) & 48 (19) & 48 (19) & 43 (14) & 48 (19) & 43 (14) & 59 (30) & 83 (54) \\ \hline
2  & 63 (34) & 44 (15) & 44 (15) & 48 (19) & 44 (15) & 43 (14) & 43 (14) & 54 (25) \\ \hline
3  & 44 (15) & 43 (14) & 45 (16) & 48 (19) & 49 (20) & 43 (14) & 48 (19) & 106(77)  \\ \hline
4  & 58 (29) & 43 (14) & 48 (19) & 53 (24) & 64 (35) & 39 (10) & 48 (19) & 63 (34) \\ \hline
5  & 79 (50) & 48 (19) & 49 (20) & 54 (25) & 44 (15) & 38 (9)  & 53 (24) & 53 (24) \\ \hline
6  & 73 (44) & 37 (9)  & 48 (19) & 39 (10) & 53 (24) & 39 (10) & 64 (35) & 53 (24) \\ \hline
7  & 88 (59) & 43 (14) & 44 (15) & 58 (29) & 49 (20) & 44 (15) & 54 (25) & 59 (30) \\ \hline
8  & 60 (31) & 48 (19) & 43 (14) & 38 (9)  & 43 (14) & 44 (15) & 73 (44) & 73 (44) \\ \hline
9  & 79 (50) & 43 (14) & 48 (19) & 53 (24) & 43 (14) & 39 (10) & 68 (39) & 58 (29) \\ \hline
10 & 65 (36) & 53 (24) & 49 (20) & 53 (24) & 49 (20) & 44 (15) & 49 (20) & 43 (14) \\ \hline
11 & 63 (34) & 48 (19) & 48 (19) & 48 (19) & 58 (29) & 43 (14) & 63 (34) & 48 (19) \\ \hline
12 & 53 (24) & 44 (15) & 43 (14) & 48 (19) & 48 (19) & 38 (9)  & 58 (29) & 53 (24) \\ \hline
13 & 48 (19) & 49 (20) & 48 (19) & 54 (25) & 49 (20) & 53 (24) & 68 (39) & 37 (9)  \\ \hline
14 & 43 (14) & 48 (19) & 48 (19) & 43 (14) & 49 (20) & 43 (14) & 78 (49) & 49 (20) \\ \hline
15 & 44 (15) & 48 (19) & 48 (19) & 48 (19) & 48 (19) & 43 (14) & 39 (10) & 73 (44) \\ \hline
16 & 58 (29) & 48 (19) & 43 (14) & 58 (29) & 37 (9)  & 38 (9)  & 58 (29) & 73 (44) \\ \hline
17 & 48 (19) & 44 (15) & 43 (14) & 53 (24) & 49 (20) & 37 (9)  & 68 (39) & 48 (19) \\ \hline
18 & 64 (35) & 43 (14) & 48 (19) & 63 (34) & 53 (24) & 38 (9)  & 93 (64) & 69 (40) \\ \hline
19 & 74 (45) & 43 (14) & 48 (19) & 58 (29) & 53 (24) & 37 (9)  & 68 (39) & 63 (34) \\ \hline
20 & 49 (20) & 43 (14) & 43 (14) & 58 (29) & 48 (19) & 44 (15) & 54 (25) & 58 (29) \\ \hline
\hline
$\floor{ \mathbf{\rm mean}}$& \textbf{60 (31)} & \textbf{45 (16)} & \textbf{46 (17)} & \textbf{50 (21)} & \textbf{48 (19)} & \textbf{41 (12)} & \textbf{60 (31)} & \textbf{60 (31)} \\ \hline
\end{tabular}
\end{table}

\begin{table}[htbp]
  \scriptsize
  \captionsetup{font=footnotesize}
  \caption{Number of function evaluations taken by 20 replications of \code{MATSuMoTo} for each of the $8$ convex test problems for $n=4$.\label{table:matsumoto4d}}
\center
\begin{tabular}{|c|r|r|r|r|r|r|r|r|}
\hline
& \prob{abhi} & \prob{quad} & \prob{KLT} & \prob{maxq} & \prob{mxhilb} & \prob{LQ} & \prob{CB3I} & \prob{CB3II} \\ \hline
1  & 110 (81) & 56 (27) & 51 (22) & 66	(37) & 60 (31) & 41 (12) & 230 (20) & 85  (56) \\ \hline
2  & 70  (41) & 55 (26) & 60 (31) & 120	(91) & 60 (31) & 46 (17) & 76  (47) & 85  (56) \\ \hline
3  & 50  (21) & 50 (21) & 50 (21) & 94	(65) & 50 (21) & 41 (12) & 150 (121)& 196 (11)  \\ \hline
4  & 104 (75) & 50 (21) & 50 (21) & 65	(36) & 80 (51) & 45 (16) & 110 (81) & 186 (16) \\ \hline
5  & 91  (62) & 60 (11) & 55 (26) & 112	(83) & 80 (51) & 41 (12) & 100 (11) & 161 (132) \\ \hline
6  & 66  (37) & 55 (26) & 56 (27) & 90	(61) & 56 (27) & 55 (26) & 101 (72) & 196 (167) \\ \hline
7  & 86  (57) & 65 (36) & 45 (16) & 65	(36) & 50 (21) & 46 (17) & 90  (61) & 141 (112) \\ \hline
8  & 90  (61) & 55 (26) & 65 (11) & 81	(52) & 65 (36) & 41 (12) & 105 (63) & 346 (161) \\ \hline
9  & 166 (137)& 56 (27) & 56 (27) & 55	(26) & 50 (21) & 51 (22) & 96  (67) & 100 (71) \\ \hline
10 & 85  (56) & 65 (21) & 55 (11) & 60	(31) & 65 (36) & 55 (26) & 55  (26) & 106 (77) \\ \hline
11 & 147 (118)& 70 (41) & 46 (17) & 70	(41) & 65 (36) & 46 (17) & 95  (66) & 80  (51) \\ \hline
12 & 96  (67) & 45 (16) & 46 (17) & 112	(83) & 71 (42) & 50 (21) & 112 (83) & 231 (202) \\ \hline
13 & 39  (11) & 60 (31) & 71 (42) & 119	(90) & 71 (42) & 41 (12) & 50  (21) & 81  (52)   \\ \hline
14 & 85  (56) & 61 (32) & 50 (21) & 85	(56) & 55 (26) & 55 (26) & 205 (176)& 55  (26) \\ \hline
15 & 101 (72) & 60 (31) & 60 (31) & 94	(65) & 55 (26) & 41 (12) & 160 (131)& 114 (85) \\ \hline
16 & 66  (37) & 51 (22) & 45 (16) & 50	(21) & 75 (46) & 41 (12) & 65  (36) & 65  (36) \\ \hline
17 & 76  (47) & 65 (36) & 56 (27) & 55	(26) & 60 (31) & 60 (31) & 85  (56) & 131 (102) \\ \hline
18 & 70  (41) & 56 (27) & 51 (22) & 110	(81) & 75 (46) & 65 (36) & 314 (285)& 70  (41) \\ \hline
19 & 81  (52) & 56 (11) & 60 (31) & 221	(192)& 71 (42) & 41 (12) & 117 (16) & 141 (52) \\ \hline
20 & 105 (76) & 46 (17) & 60 (31) & 66	(37) & 50 (21) & 56 (27) & 210 (13) & 50  (21) \\ \hline
\hline
$\floor{ \mathbf{\rm mean}}$& \textbf{89 (60)} & \textbf{56 (25)} & \textbf{54 (23)} & \textbf{89 (60)} & \textbf{63 (34)} & \textbf{47 (18)} & \textbf{126 (81)} & \textbf{131 (76)} \\ \hline
\end{tabular}
\end{table}

\begin{table}[htbp]
  \scriptsize
  \captionsetup{font=footnotesize}
  \caption{Number of function evaluations taken by 20 replications of \code{MATSuMoTo} for each of the $8$ convex test problems for $n=5$.\label{table:matsumoto5d}}
\center
\begin{tabular}{|c|r|r|r|r|r|r|r|r|}
\hline
& \prob{abhi} & \prob{quad} & \prob{KLT} & \prob{maxq} & \prob{mxhilb} & \prob{LQ} & \prob{CB3I} & \prob{CB3II} \\ \hline
1  & 472 (443) & 82  (53) & 83	(54) & 137 (108) & 149 (120) & 52 (23) & 199 (170) & 83  (54) \\ \hline
2  & 358 (329) & 83  (54) & 317	(288)& 229 (200) & 208 (179) & 63 (34) & 267 (238) & 518 (489) \\ \hline
3  & 193 (164) & 83  (54) & 172	(143)& 115 (86) & 175 (146) & 62 (33) & 377 (348) & 1,626 (133)  \\ \hline
4  & 194 (165) & 73  (44) & 62	(33) & 171 (142) & 72  (43) & 57 (28) & 223 (194) & 325 (296) \\ \hline
5  & 160 (131) & 72  (43) & 73	(44) & 504 (475) & 87  (58) & 53 (24) & 835 (806) & 487 (458) \\ \hline
6  & 396 (367) & 187 (127) & 78	(49) & 92  (63) & 152 (123) & 53 (24) & 127 (98) & 196 (18) \\ \hline
7  & 82  (53) & 82  (53) & 120	(91) & 135 (106) & 247 (218) & 62 (33) & 184 (155) & 603 (574) \\ \hline
8  & 62  (33) & 106 (15) & 58	(29) & 969 (940)  & 186 (157) & 52 (23) & 164 (135) & 67  (38) \\ \hline
9  & 107 (78) & 83  (54) & 68	(39) & 276 (247) & 103 (74) & 57 (28) & 399 (370) & 671 (642) \\ \hline
10 & 186 (13) & 78  (18) & 73	(44) & 352 (323) & 82  (53) & 73 (44) & 364 (335) & 205 (18) \\ \hline
11 & 181 (152) & 130 (65) & 62	(33) & 255 (226) & 53  (24) & 58 (29) & 414 (385) & 240 (211) \\ \hline
12 & 239 (210) & 115 (86) & 125	(96) & 403 (374) & 181 (152) & 47 (18) & 83  (54) & 462 (433) \\ \hline
13 & 295 (211) & 87  (18) & 72	(43) & 167 (138) & 151 (122) & 48 (19) & 128 (99) & 324 (33)   \\ \hline
14 & 354 (325) & 78  (49) & 72	(43) & 236 (207) & 82  (53) & 62 (33) & 200 (171) & 154 (125) \\ \hline
15 & 113 (13) & 73  (44) & 52	(14) & 512 (483) & 88  (59) & 53 (24) & 271 (242) & 87  (58) \\ \hline
16 & 227 (23) & 82  (53) & 120	(91) & 189 (160) & 93  (64)   & 68 (39) & 326 (297) & 63  (34) \\ \hline
17 & 102 (73) & 135 (106)& 270	(241)& 195 (166) & 57  (28) & 58 (29) & 180 (151) & 599 (570) \\ \hline
18 & 139 (110) & 268 (80) & 150	(121)& 173 (144) & 112 (83) & 58 (29) & 228 (199) & 112 (83) \\ \hline
19 & 169 (19) & 67  (38) & 72	(43) & 108 (79) & 117 (88) & 42 (13) & 204 (175) & 93  (64) \\ \hline
20 & 268 (239) & 301 (195)& 62	(13) & 465 (436) & 237 (208) & 53 (24) & 162 (133) & 180 (151) \\ \hline
\hline
$\floor{ \mathbf{\rm mean}}$& \textbf{214 (157)} & \textbf{113 (62)} & \textbf{108 (77)} & \textbf{284 (255)} & \textbf{131 (102)} & \textbf{56 (27)} & \textbf{266 (237)} & \textbf{281 (224)} \\ \hline
\end{tabular}
\end{table}

\FloatBarrier 
\section{Derivation of Tolerances}\label{app:tolerance} \setcounter{table}{0}
In this section, we show how one can compute sufficient values of $M_{\eta}$,
$\epsilon_{\lambda}$, and $M_{\lambda}$ for the MILP model \eqref{model:CPFull}.
We first show that if we choose these parameters incorrectly, then the
resulting MILP model no longer provides a valid lower bound.

\paragraph{Effect of an Insufficient $\epsilon_\lambda$ Value.}\label{epsilonFix}
A large $\epsilon_\lambda$ or small $M_\lambda$ (or both) could result in an incorrect value of
$z^{i_j}=1$ for a point $x \notin \conep{x^{i_j}-X^{\bi}}$, violating the implication of $z^{i_j}=1$ and
yielding an invalid lower bound on $f$. This is illustrated using a
one-dimensional example in \figref{epsilon}.
Similarly, one can encounter an invalid lower bound at an iteration if
$M_\lambda$ is chosen to be smaller than required.

\begin{figure}[htpb]
\captionsetup{font=footnotesize}
\center
\includegraphics[width=0.38\textwidth]{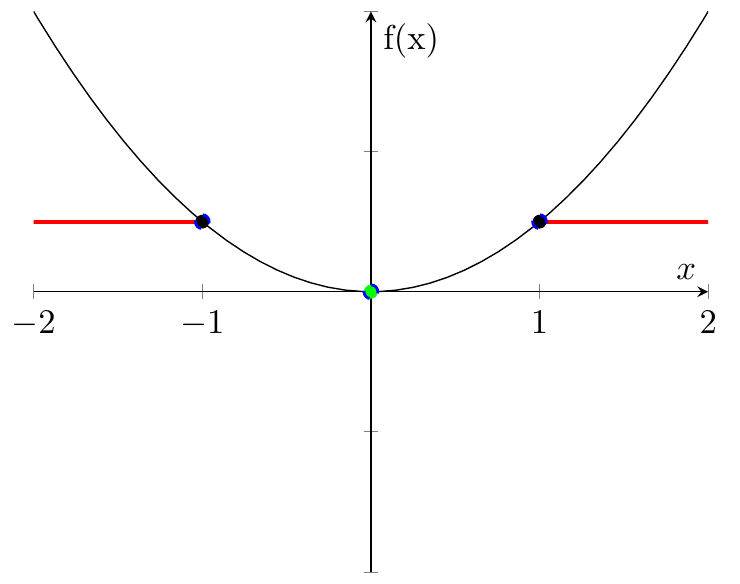}
\caption{An example of false termination of \algref{framework} using
\eqref{model:CPFull} when an insufficient value of $\epsilon_\lambda$ is used:
$f(x)=x^2$ and interpolation points $-1$ and $1$ are used to form the secant
shown in red colour. Any value of $\epsilon_\lambda > 0.5$ forces
$z^{{1}_1}=z^{{1}_2}=1$, activating the cut $\eta \ge 1$ at the optimal $x^*=0$,
resulting in $l_f^k=u_f^k=1$.\label{fig:epsilon}}
\end{figure}

\paragraph{Bound on $M_{\eta}$} Let $l_f$ be a valid lower bound of $f$ on $\Omega$.
With this lower bound, scalars $M_{\bi}$ can be defined as
\begin{equation}\label{eq:bigMi1}
M_{\bi}=\max\limits_{x \in \Omega}\{(c^{\bi})^T x +
b^{\bi}\} - l_f, \; \forall \bi \in W(X).
\end{equation}
If $\Omega= \left\{ x: l_x \le x \le u_x \right\}$, where $l_x, u_x \in \R^n$
are known, then we can set
\begin{equation}\label{eq:bigMi2}
M_{\bi} = \sum_{h:c^{\bi}_j<0}c^{\bi}_h l_x +
\sum_{h:c^{\bi}_j \ge 0}c^{\bi}_h u_x + b^{\bi} - l_f, \quad h = 1,
\ldots, n, \; \forall \bi \in W(X).
\end{equation}
Then, we can either set individual values of $M_{\bi}$ within each constraint
\eqref{eq:cuts}, which would yield a tighter model, or use a single parameter,
\begin{equation}\label{eq:bigMi3}
M_{\eta}=\max\limits_{\bi} \{M_{\bi}\},
\end{equation}
as shown in \eqref{model:CPFull}.

\paragraph{Bounds on $\epsilon_\lambda$ and $M_\lambda$}
Sufficient values of $M_\lambda$ and $\epsilon_\lambda$ are not easy
to calculate as they depend on the rays generated at $x^{i_j}$,
and the domain $\Omega$. We show next, that bounds on these values can be
computed by solving a set of optimization problems.
A sufficient value for $\epsilon_\lambda$ can be computed based on the
representations of the $n+1$ hyperplanes formed using different combinations of
$n$ points, $\forall \bi \in W(X)$. We can use one of the representations of
the hyperplane passing through points $X^{i}\setminus \{x^{i_j}\}, j=1\ldots n+1$
to obtain bounds on $\epsilon_\lambda$, by solving an optimization problem for
each $i_j \in \bi$. Let $c^{\biminusj}$ and $b^{\biminusj}$ denote the
solution of the following problem.
\begin{equation} \label{model:hyper} \tag{\text{P$-hyp$}}
\begin{aligned}
\dps \mini_{c, b} & \left\| c \right\|_1\\
\st  & \dps (c)^T x^{i_l} + b = 0, \quad l = 1, \ldots, n+1,\; l \neq j \\
     & \dps \left\| c \right\|_1 \ge 1 \\
     & \dps c \in \Z^n, b \in \Z .
\end{aligned}
\end{equation}
Problem \eqref{model:hyper} can be easily cast as an integer program by
replacing $\left\|c \right\|_1$ by $e^Ty$, where $e=(1,\ldots,1)^T$ is the
vector or all ones, and the variables $y$ satisfy the constraints $y \geq c, y
\geq - c$. Because, the hyperplane is generated using $n$ integer points, it
can be shown that there exists a solution $c^{\biminusj}$ and $b^{\biminusj}$
that is nonzero and integral, as stated in the following proposition.

\begin{proposition}\label{prop:distance}
Consider a hyperplane $S\defined\{x:c^{T} x + b = 0\}$ such that $c$ and $b$ are
integral. The Euclidean distance between an arbitrary point $\hat{x} \in \Z^n
\setminus S$, and $S$, is greater than or equal to $\frac{1}{\lvert \lvert c
\rvert \rvert_2}$.
\end{proposition}

\begin{proof}
  The Euclidean distance between a point $\hat{x}$ and $S$ is the
  2-norm of the projection of the line segment joining an arbitrary point $w \in
  S$ and $\hat{x}$, on the normal passing through $\hat{x}$, and can be expressed as
  \begin{equation}
    \frac{\lvert c^{T} \hat{x} + b \rvert}{\lvert \lvert c \rvert \rvert_2} .
  \end{equation}
  By integrality of $c, b$ and $\hat{x}$, and since $\hat{x} \notin S$, $\lvert c^{T} \hat{x} + b
  \rvert \geq 1$, and the result follows.
  \qed
\end{proof}
\propref{distance} can be used directly to get a sufficient value of
$\epsilon_\lambda$ as follows.
\begin{equation}\label{eq:suffEps0}
\epsilon_\lambda = \min\limits_{\bi \in W(X),\; i_j \in \bi}{\; \frac{1}{\lvert
\lvert c^{\biminusj} \rvert \rvert_2}} .
\end{equation}
The bound can be tightened using the following optimization problem.
\begin{equation} \tag{P-\text{$\epsilon$}} \label{model:eps}
\begin{aligned}
\dps\mini_{x} & \frac{\lvert (c^{\biminusj})^T x + b^{\biminusj} \rvert}{\lvert
\lvert c^{\biminusj} \rvert \rvert_2} \\
\st  & \dps \lvert(c^{\biminusj})^T x^{i_l} + b^{\biminusj}\rvert \geq 1 \\
     & \dps x \in \Omega.\\
\end{aligned}
\end{equation}
If we denote by $\epsilon^{\biminusj}_\lambda$, the optimal value of \eqref{model:eps},
then the following is a sufficient value of $\epsilon_\lambda$.
\begin{equation}\label{eq:suffEps}
\epsilon_\lambda = \min\limits_{\bi \in W(X),\; i_j \in \bi}{\; \epsilon_\lambda^{\biminusj}}.
\end{equation}
Similarly, if we maximize the objective in \eqref{model:eps}, and denote the optimal
value by $M^{\biminusj}_\lambda$ we get a sufficient value for $M_\lambda$ as follows.
\begin{equation}\label{eq:suffbigM}
M_\lambda = \max\limits_{\bi \in W(X),\; i_j \in \bi}{\; M_\lambda^{\biminusj}}.
\end{equation}

\paragraph{No-Good Cuts and Stronger Objective Bounds.} For our initial
experiments, we tried setting arbitrarily small values for
$\epsilon_{\lambda}$ instead of obtaining the best possible values by solving
a set of optimization problems as elaborated above.
The problem with having such values of $\epsilon_{\lambda}>0$ too small is
that it allows us to ignore the cuts at the interpolation points, $x^{(k)}$,
causing a violation of the bound, $\eta \geq f^{(k)}$, for $x=x^{(k)}$ in the
MILP, resulting in repetition of iterates in the algorithm. We fixed this
by adding the following valid inequality to the MILP.
\[ \eta \geq f^{(k)} - M \| x^{(k)} - x \|_1 .\]
We can write this inequality as a linear constraint, by introducing a binary
representation of the variables, $x$, requiring $n U$ binary variables,
$\xi_{ij}, \; i=1,\ldots,n, \; j=1,\ldots,U$, where $n$ is the dimension of
our problem, and $0 \leq x \leq U$. With this
representation, we obtain \[ x_i = \sum_{j=1}^U i \xi_{ij}, \quad 1=
\sum_{j=1}^U \xi_{ij}, \quad \xi_{ij} \in \{0,1\}, \] and write the
$\eta$-constraint equivalently as
\[ \eta \geq f^{(k)} - M \sum_{i=1}^n \left(  \sum_{j: \xi_{ij}^{(k)}=0} \xi_{ij}
				   + \sum_{j: \xi_{ij}^{(k)}=1} \left( 1 - \xi_{ij} \right) \right) .
\]

 
\clearpage

\bibliographystyle{spmpsci}
\bibliography{../../bibs/merged}

\vfill
\begin{flushright}
\scriptsize
\framebox{\parbox{\textwidth}{
The submitted manuscript has been created by UChicago Argonne, LLC, Operator of Argonne National Laboratory (“Argonne”). 
Argonne, a U.S. Department of Energy Office of Science laboratory, is operated under Contract No. DE-AC02-06CH11357. 
The U.S. Government retains for itself, and others acting on its behalf, a paid-up nonexclusive, irrevocable worldwide 
license in said article to reproduce, prepare derivative works, distribute copies to the public, and perform publicly 
and display publicly, by or on behalf of the Government.  The Department of Energy will provide public access to these 
results of federally sponsored research in accordance with the DOE Public Access Plan. 
\url{http://energy.gov/downloads/doe-public-access-plan}
}}
\normalsize
\end{flushright}

\end{document}